\documentclass[11pt]{amsart}

\usepackage{amssymb}
\usepackage[a4paper]{geometry}
\usepackage[parfill]{parskip} % to avoid to indent lines
\usepackage{color}
\usepackage{url}
\usepackage{graphicx}
\usepackage{url}
\usepackage{float}

\usepackage{tikz}
\newcommand{\cP}{\mathcal{P}}

\newcommand{\cN}{\mathcal{N}}

\newcommand{\N}{\mathbb{N}}

\newcommand{\leo}{}

\newtheorem{remark}{Remark}
\newtheorem{definition}{Definition}
\newtheorem{theorem}{Theorem}
\newtheorem{lemma}{Lemma}
\newtheorem{observation}{Observation}
\newtheorem{proposition}{Proposition}
\newtheorem{corollary}{Corollary}

\title{Leaf-reconstructibility of phylogenetic networks}
\author{Leo van Iersel}
  \address[Leo van Iersel]{Delft Institute of Applied Mathematics, Delft University of Technology, The Netherlands}
\author{Vincent Moulton}
  \address[Vincent Moulton]{School of Computing Sciences,  University of East Anglia, Norwich, United Kingdom}
\email[Leo van Iersel]{l.j.j.v.iersel@gmail.com}
\email[Vincent Moulton]{v.moulton@uea.ac.uk}

\begin{document}

\maketitle

% REQUIRED
\begin{abstract}
An important problem in evolutionary biology is to reconstruct the 
evolutionary history of a set $X$ of species. This history is often
represented as a phylogenetic network, that is, a connected graph with leaves labelled by 
elements in $X$ (for example, an evolutionary tree), which is usually also 
binary, i.e. all vertices have degree 1 or 3. A common approach
used in phylogenetics to build a phylogenetic network on $X$ involves constructing it from
networks on subsets of $X$. Here we \leo{consider
the question of which (unrooted) phylogenetic networks are {\em leaf-reconstructible}, i.e. 
which networks can be uniquely reconstructed from the set of networks obtained from it by deleting a single leaf} (its {\em $X$-deck}). 
This problem is closely related to
the (in)famous reconstruction conjecture in graph theory but, as we shall show, 
presents distinct challenges. We show that some large classes of phylogenetic networks
are reconstructible from their $X$-deck. This include\leo{s} phylogenetic trees, 
binary networks \leo{containing at least one
non-trivial cut-edge}, and \leo{binary level-4} networks (the 
level of a network measures how far it is from being a tree). We also show that for fixed $k$, 
almost all binary level-$k$ phylogenetic networks are leaf-reconstructible. As an application
of our results, we show that a level-3 network $N$ can be reconstructed from its
quarnets, that is, 4-leaved networks that are induced by $N$ in a certain
recursive fashion. Our results lead to several interesting open problems which 
we discuss, including the conjecture that all phylogenetic networks
with at least five leaves are leaf-reconstructible.
\end{abstract}
%
%% REQUIRED
%\begin{keywords}
%  phylogenetic trees, phylogenetic networks, graph reconstruction, reconstruction conjecture
%\end{keywords}
%
%% REQUIRED
%\begin{AMS}
%  05C60, 92D15
%\end{AMS}

\section{Introduction}

An important problem in evolutionary biology is to reconstruct the 
evolutionary history of a set of species. This commonly involves 
constructing some form of phylogenetic network, that is, 
a graph (often a tree) labeled by a set $X$ of species, for which
some data (e.g. molecular sequences) has been collected. Over the 
past four decades several ways have been introduced to 
construct phylogenetic trees (see e.g. \cite{felsenstein}) and, more recently, methods have been
developed to construct more  general phylogenetic networks (see e.g. \cite{gusfield,huson}).

One particular approach for constructing phylogenetic networks involves
building them up from smaller networks. This approach is 
particularly useful when it is only 
feasible to compute networks from the biological data on small datasets (e.g. when using likelihood approaches).
The problem of building trees from smaller trees has been studied for \leo{some time} (where it is commonly known as the supertree problem; cf. e.g. \cite[Chapter 6]{SS})
but \leo{the related problem for networks} has been only considered more recently (see e.g. \cite{information,huber2011encoding} \leo{focussing on directed phylogenetic networks and~\cite{thatte06} focussing on pedigrees}).
Even so, this problem can be extremely challenging. 

In this paper, we shall present a unified approach to 
constructing phylogenetic networks from smaller networks. 
We shall \leo{consider \emph{unrooted}} phylogenetic networks (cf. \cite{gambette}).
Essentially, these are connected  graphs with leaf-set labelled by a set $X$; they are
called \leo{\emph{binary}} if the degree of every vertex is 1 or 3.
For such networks, we focus on the problem of reconstructing a phylogenetic network from
its {\em $X$-deck}, \leo{roughly speaking, this} is the collection of \leo{networks that 
is obtained by deleting one leaf and} supressing 
the resulting degree\leo{-}2 vertex. We call a network that can be reconstructed from 
its $X$-deck {\em leaf-reconstructible}. \leo{See Sections~\ref{sec:preliminaries} and~\ref{sec:xdeck} for formal definitions.}

%Note that a similar problem has been considered to reconstructing pedigrees from subpedigrees \cite{thatte06}, although this leads to different challenges.

Intriguingly, the problem of reconstructing a graph from 
its vertex deleted subgraphs has been studied for over 75 years (it 
was introduced in 1941 by Kelly and Ulam \cite{bh77}), where 
it is known as the {\em reconstruction conjecture}. In particular, 
this conjecture states that every finite simple undirected graph on three of more vertices
can be constructed from its collection of vertex deleted subgraphs.
\leo{This conjecture remains} open, but has been
shown to hold for several large and important classes of graphs \cite{bh77}.
Even so, as we shall see, although \leo{determining} leaf-reconstructibilty of a phylogenetic
network is closely related to the reconstruction conjecture, there 
are several key differences which mean that they need to be 
treated as quite distinct problems.

We now summarize the contents of the rest of the paper. In the
next section, we present some preliminaries concerning phylogenetic networks.
In Section~\ref{sec:xdeck}, we then formally define leaf-reconstructibility and explain why 
this concept is distinct from the notion of endvertex reconstructibilty 
a well-studied concept in graph reconstruction theory (see \cite[p.237]{bh77}).
In addition, we show that certain key features of a binary phylogenetic network (such
as its level and reticulation number) can be \leo{reconstructed} from its $X$-deck.

In Section~\ref{sec:decomp}, we then show that a large class of phylogenetic networks, which
we call {\em decomposable networks} are leaf-reconstructible. These are networks
\leo{containing at least one cut-edge not incident to a leaf}. To show this we first show
that any phylogenetic tree with at least 5 leaves is \leo{leaf-}reconstructible. We also note that \leo{phylogenetic trees with 4 leaves are} not \leo{leaf-}reconstructible. Our result concerning decomposable 
networks is analogous to a result by Yongzhi \cite{y88} who showed that
the graph reconstruction conjecture can be restricted to considering 2-connected graphs.

The fact that decomposable networks are reconstructible 
implies that we can restrict our attention to leaf-reconstructibility of {\em simple} networks,
that is, non-decomposable networks. An important feature of a phylogenetic network $N$  is 
its {\em level}, which measures how far away the network is from being
a phylogenetic tree (in particular, trees \leo{are} level-0 networks). By considering 
certain subconfigurations in simple networks, in Section~\ref{sec:simple}, we prove that, for fixed $k$, 
almost all level-$k$ networks are leaf-reconstructible. 

In Section~\ref{sec:number}, we then turn to the problem of computing the
smallest number of elements in the $X$-deck of a leaf-reconstructible 
\leo{network} that are required \leo{to} reconstruct \leo{it}, which we call \leo{its} 
leaf-reconstruction number. This is analogous to the so-called
reconstruction number of a graph (cf. \cite{aflm} for a survey on these numbers). 
In particular, we show that the leaf-reconstruction
number of any phylogenetic tree on 5 or more leaves is 2, unless
it is a star-tree in which case \leo{this number} is 3. We also show that this implies 
that the leaf-reconstruction
number of any decomposable phylogenetic network with at least 5 leaves
is 2.

In Section~\ref{sec:low}, we turn our attention to low-level networks, 
showing that \leo{all binary level-4 networks with at least five leaves have leaf-reconstruction number at most~2. The proof uses several lemmas that could be useful in studying the leaf-reconstructibility of higher-level networks.}

In practice, most methods for constructing
phylogenetic networks from smaller networks to date have focussed on using 
networks with small numbers of leaves (in the rooted \leo{case}, often 3-leaved networks). In
Section~\ref{sec:quarnets}, by using a recursive argument and our previous results, we show that
any level-3 network can be reconstructed from its set of {\em quarnets}.
Essentially, these are 4-leaved networks which are obtained from $N$ 
by selecting 4 leaves in the network\leo{, removing all other leaves and suppressing degree-2 vertices, multi-edges and biconnected components with two incident cut-edges}.  Our result on quartnets
is analogous to results presented in \cite{level2trinets} for level-2 rooted phylogenetic networks.

Several variants of the reconstruction conjecture have been considered
in the literature (see \cite{bh77}). We can also consider variants for
phylogenetic networks. In Section~\ref{sec:idge}, we consider the problem
of reconstructing a phylogenetic network from its collection of edge-deleted
subgraphs, showing that \leo{in this setting} we can sharpen the leaf-\leo{reconstructibility} bounds that 
we \leo{previously} obtained. We then conclude in the last section
by discussing the problem of reconstructing directed phylogenetic 
networks, as well as various open problems.

\section{Preliminaries}\label{sec:preliminaries}

In this section, we present some preliminaries concerning phylogenetic networks (cf.  \cite{gambette})

Let~$X$ be a finite set with~$|X|\geq 2$.

\begin{definition}
A \emph{phylogenetic tree} on~$X$ is a tree with no degree-2 vertices in which the leaves (degree-1 vertices) are bijectively labelled by the elements of~$X$. 
\end{definition}

A \emph{biconnected component} of a graph is a maximal 2-connected subgraph and it is called a \emph{blob} if it contains at least two edges.

\begin{definition}
A \emph{phylogenetic network} on~$X$ is a connected graph~$N$ such that contracting each blob (one by one) into a single vertex gives a phylogenetic tree on~$X$. 
\end{definition}

A bipartition $A|B$ of~$X$, with~$A,B\neq\emptyset$ is a \emph{split} of a phylogenetic network~$N$ if~$N$ contains a cut-edge~$e$ such that the elements of~$A$ and~$B$ are the leaf-labels of the two connected components of~$N-e$. If this is the case, we also say that the split $A|B$ is \emph{induced} by~$e$. From the definition of a phylogenetic network it follows that each of its cut-edges induces a split and no two cut-edges induce the same split. Moreover, the phylogenetic tree obtained by contracting each blob of~$N$ into a single vertex is the unique phylogenetic tree that has precisely the same splits as~$N$. This phylogenetic tree is denoted~$T(N)$, see Figure~\ref{fig:prelim} for an example.

A cut-edge is called \emph{trivial} if at least one of its endpoints is a leaf. A phylogenetic network with at least one nontrivial cut-edge is called \emph{decomposable}. We call a phylogenetic network \emph{simple} if it has precisely one blob. 

\begin{definition}
A \emph{pseudo-network} on~$X$ is a multigraph with no degree-2 vertices in which the leaves (degree-1 vertices) are bijectively labelled by the elements of~$X$.
\end{definition}

\leo{Hence, each phylogenetic tree is a phylogenetic network and each phylogenetic network is a pseudo-network.} We let~$L(N),V(N),E(N)$ denote, respectively, the set of leaves, vertices and edges of \leo{a pseudo-network~$N$. In addition, the} phylogenetic tree~$T(N)$ is defined as the phylogenetic tree obtained by contracting each blob of~$N$ into a single vertex and suppressing any resulting degree-2 vertices. Two pseudo-networks~$N,N'$ are \emph{equivalent}, denoted~$N\sim N'$ if there exists a graph isomorphism between~$N$ and~$N'$ that is the identity on~$X$.

A \leo{pseudo-network} is called \emph{binary} if 
every non-leaf vertex has degree~3.  Note that our definition of
a binary phylogenetic network is slightly different from the one presented in \cite{gambette}, 
and has the advantage that for fixed $X$, there are only finitely many phylogenetic 
networks with \leo{fixed level and} leaf-set $X$ (essentially because the number of phylogenetic
trees with leaf set $X$ is finite cf. \cite{SS}).
Note also that a binary phylogenetic network is simple precisely when it is not decomposable \leo{and not a star tree}. However, 
this is not the case for nonbinary networks.

\section{$X$-decks and leaf-reconstructibility}\label{sec:xdeck}

In this section we introduce the concept of leaf-reconstructibility. We begin
by defining the $X$-deck for a phylogenetic network on $X$. 

Given a phylogenetic network~$N$ and a vertex~$v\in V(N)$, the pseudo-network~$N_v$ is the result of deleting vertex~$v$ from~$N$, together with its incident edges, and suppressing resulting degree-2 vertices. See Figure~\ref{fig:prelim} for an example. Given a phylogenetic 
network~$N$ on~$X$ and~$U\subseteq V(N)$, the $U$-\emph{deck} of~$N$ is the multiset $\{N_u \mid u\in U\}$. 

\begin{figure}
\centering
%\centerline{\includegraphics{fig/fig1}}
 \begin{tikzpicture}[scale=.8]
	 \tikzset{lijn/.style={very thick}}
	 
	 % N
	 \begin{scope}[xshift=0cm,yshift=0cm]
	\draw[very thick, fill, radius=0.06] (-.5,1) circle node[left] {$b$};
	\draw[very thick, fill, radius=0.06] (1,-.5) circle node[below] {$c$};
	\draw[very thick, fill, radius=0.06] (1,2.5) circle node[above] {$a$};
	\draw[very thick, fill, radius=0.06] (4,-.5) circle node[below] {$d$};
	\draw[very thick, fill, radius=0.06] (5.5,2.5) circle node[above right] {$e$};
	\draw[very thick, fill, radius=0.06] (0,1) circle;
	\draw[very thick, fill, radius=0.06] (1,0) circle;
	\draw[very thick, fill, radius=0.06] (1,2) circle;
	\draw[very thick, fill, radius=0.06] (2,1) circle;
	\draw[very thick, fill, radius=0.06] (3,1) circle;
	\draw[very thick, fill, radius=0.06] (4,0) circle;
	\draw[very thick, fill, radius=0.06] (5,1) circle;
	\draw[very thick, fill, radius=0.06] (4,2) circle;
	\draw[very thick, fill, radius=0.06] (5,2) circle;
	\draw[lijn] (1,0) -- (0,1);
	\draw[lijn] (1,0) -- (2,1);
	\draw[lijn] (0,1) -- (1,2);
	\draw[lijn] (1,2) -- (2,1);
	\draw[lijn] (2,1) -- (3,1);
	\draw[lijn] (3,1) -- (4,2);
	\draw[lijn] (3,1) -- (4,0);
	\draw[lijn] (4,0) -- (5,1);
	\draw[lijn] (4,2) -- (5,1);
	\draw[lijn] (4,2) -- (5,2);
	\draw[lijn] (5,1) -- (5,2);
	\draw[lijn] (0,1) -- (-.5,1);
	\draw[lijn] (1,0) -- (1,-.5);
	\draw[lijn] (1,2) -- (1,2.5);
	\draw[lijn] (4,0) -- (4,-.5);
	\draw[lijn] (5,2) -- (5.5,2.5);
	\draw (2.5,-1.5) node {$N$};
	\end{scope}
	
	% T(N)
	 \begin{scope}[xshift=0cm,yshift=-5cm]
	\draw[very thick, fill, radius=0.06] (1,1) circle node[left] {$b$};
	\draw[very thick, fill, radius=0.06] (2,0) circle node[below] {$c$};
	\draw[very thick, fill, radius=0.06] (2,2) circle node[above] {$a$};
	\draw[very thick, fill, radius=0.06] (4,0) circle node[below] {$d$};
	\draw[very thick, fill, radius=0.06] (4,2) circle node[above right] {$e$};
	\draw[very thick, fill, radius=0.06] (2,1) circle;
	\draw[very thick, fill, radius=0.06] (3,1) circle;
	\draw[lijn] (1,1) -- (2,1);
	\draw[lijn] (2,1) -- (2,2);
	\draw[lijn] (2,1) -- (2,0);
	\draw[lijn] (2,1) -- (3,1);
	\draw[lijn] (3,1) -- (4,0);
	\draw[lijn] (3,1) -- (4,2);
	\draw (2.5,-1.5) node {$T(N)$};
	\end{scope}
	
	% N_a
	 \begin{scope}[xshift=8cm,yshift=0cm]
	\draw[very thick, fill, radius=0.06] (-.5,1) circle node[left] {$b$};
	\draw[very thick, fill, radius=0.06] (1,-.5) circle node[below] {$c$};
	\draw[very thick, fill, radius=0.06] (4,-.5) circle node[below] {$d$};
	\draw[very thick, fill, radius=0.06] (5.5,2.5) circle node[above right] {$e$};
	\draw[very thick, fill, radius=0.06] (0,1) circle;
	\draw[very thick, fill, radius=0.06] (1,0) circle;
	\draw[very thick, fill, radius=0.06] (2,1) circle;
	\draw[very thick, fill, radius=0.06] (3,1) circle;
	\draw[very thick, fill, radius=0.06] (4,0) circle;
	\draw[very thick, fill, radius=0.06] (5,1) circle;
	\draw[very thick, fill, radius=0.06] (4,2) circle;
	\draw[very thick, fill, radius=0.06] (5,2) circle;
	\draw[lijn] (1,0) -- (0,1);
	\draw[lijn] (1,0) -- (2,1);
	\draw[lijn] (0,1) -- (2,1);
	\draw[lijn] (2,1) -- (3,1);
	\draw[lijn] (3,1) -- (4,2);
	\draw[lijn] (3,1) -- (4,0);
	\draw[lijn] (4,0) -- (5,1);
	\draw[lijn] (4,2) -- (5,1);
	\draw[lijn] (4,2) -- (5,2);
	\draw[lijn] (5,1) -- (5,2);
	\draw[lijn] (0,1) -- (-.5,1);
	\draw[lijn] (1,0) -- (1,-.5);
	\draw[lijn] (4,0) -- (4,-.5);
	\draw[lijn] (5,2) -- (5.5,2.5);
	\draw (2.5,-1.5) node {$N_a$};
	\end{scope}

	% N_e
	 \begin{scope}[xshift=8cm,yshift=-5cm]
	\draw[very thick, fill, radius=0.06] (-.5,1) circle node[left] {$b$};
	\draw[very thick, fill, radius=0.06] (1,-.5) circle node[below] {$c$};
	\draw[very thick, fill, radius=0.06] (1,2.5) circle node[above] {$a$};
	\draw[very thick, fill, radius=0.06] (4,-.5) circle node[below] {$d$};
	\draw[very thick, fill, radius=0.06] (0,1) circle;
	\draw[very thick, fill, radius=0.06] (1,0) circle;
	\draw[very thick, fill, radius=0.06] (1,2) circle;
	\draw[very thick, fill, radius=0.06] (2,1) circle;
	\draw[very thick, fill, radius=0.06] (3,1) circle;
	\draw[very thick, fill, radius=0.06] (4,0) circle;
	\draw[very thick, fill, radius=0.06] (5,1) circle;
	\draw[very thick, fill, radius=0.06] (4,2) circle;
	\draw[lijn] (1,0) -- (0,1);
	\draw[lijn] (1,0) -- (2,1);
	\draw[lijn] (0,1) -- (1,2);
	\draw[lijn] (1,2) -- (2,1);
	\draw[lijn] (2,1) -- (3,1);
	\draw[lijn] (3,1) -- (4,2);
	\draw[lijn] (3,1) -- (4,0);
	\draw[lijn] (4,0) -- (5,1);
	\draw[lijn] (0,1) -- (-.5,1);
	\draw[lijn] (1,0) -- (1,-.5);
	\draw[lijn] (1,2) -- (1,2.5);
	\draw[lijn] (4,0) -- (4,-.5);
	\draw[lijn] (4,2) to[out=-20,in=110] (5,1);
	\draw[lijn] (4,2) to[out=-70,in=160] (5,1);	
	\draw (2.5,-1.5) node {$N_e$};
	\end{scope}

	\end{tikzpicture}
\caption{\label{fig:prelim} A binary phylogenetic network~$N$, the phylogenetic tree~$T(N)$, and two elements of the $X$-deck of~$N$: the phylogenetic network~$N_a$ and the pseudo-network~$N_e$.}
\end{figure}
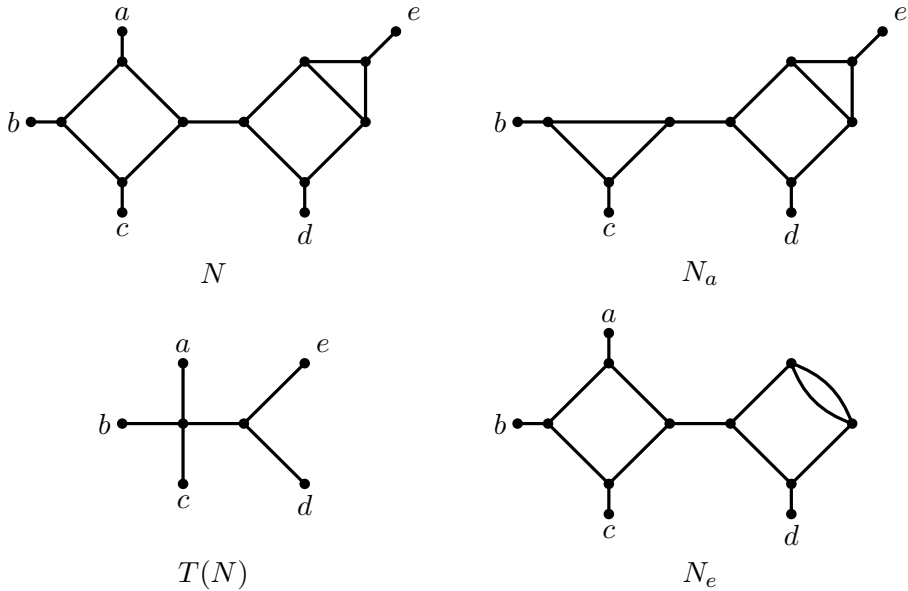

A \emph{$U$-reconstruction} of a network~$N$ on~$X$ is a network~$N'$ on~$X$ with~$V(N')=V(N)$ and~$N'_u\sim N_u$ for all~$u\in U$.
We call a phylogenetic network~$N$ $U$-\emph{reconstructible} if every $U$-reconstruction of~$N$ is equivalent to~$N$.  The $U$-\emph{reconstruction number} of a network~$N$ on~$X$ is the smallest~$k$ for which there is a subset $U'\subseteq U$ with~$|U'|=k$ such that~$N$ is~$U'$-reconstructible.

We are usually interested in the case that $U\subseteq X$. For the case that~$U=X$, we will also refer to $X$-reconstruction, $X$-reconstructible and $X$-reconstruction number as \emph{leaf-reconstruction}, \emph{leaf-reconstructible} and \emph{leaf-reconstruction number}, respectively.
It could also be interesting to take $U=V(N)$, but we shall not consider this possibility in this paper. 

If~$N$ is a binary network on~$X$ and~$x\in X$ then~$N$ can 
be obtained from~$N_x$ by \emph{attaching}~$x$ to some edge~$e$, i.e., to 
subdivide~$e$ by a new vertex~$v$ and adding a vertex labelled~$x$ and an edge between~$v$ and~$x$.
For example, the network~$N$ in Figure~\ref{fig:prelim} is $\{e\}$-reconstructible since it can be uniquely reconstructed from~$N_e$ by attaching leaf~$e$ to one of the multi-edges. Hence, this network has leaf-reconstruction number~1. The networks in Figure~\ref{fig:notleafreconstructible} are 
not leaf-reconstructible since both networks have the same~$X$-deck.

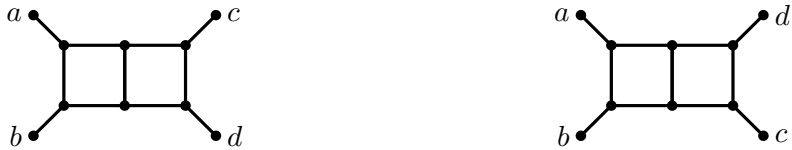
\begin{figure}
\centering
%\centerline{\includegraphics{fig/fig2}}
 \begin{tikzpicture}[scale=.8]
	 \tikzset{lijn/.style={very thick}}
	 
	 % N_1
	 \begin{scope}[xshift=0cm,yshift=0cm]
	\draw[very thick, fill, radius=0.06] (-.5,1.5) circle node[left] {$a$};
	\draw[very thick, fill, radius=0.06] (-.5,-.5) circle node[left] {$b$};
	\draw[very thick, fill, radius=0.06] (2.5,1.5) circle node[right] {$c$};
	\draw[very thick, fill, radius=0.06] (2.5,-.5) circle node[right] {$d$};
	\draw[very thick, fill, radius=0.06] (0,0) circle;
	\draw[very thick, fill, radius=0.06] (1,0) circle;
	\draw[very thick, fill, radius=0.06] (0,1) circle;
	\draw[very thick, fill, radius=0.06] (1,1) circle;
	\draw[very thick, fill, radius=0.06] (2,1) circle;
	\draw[very thick, fill, radius=0.06] (2,0) circle;
	\draw[lijn] (0,0) -- (0,1);
	\draw[lijn] (0,0) -- (1,0);
	\draw[lijn] (0,1) -- (1,1);
	\draw[lijn] (1,0) -- (1,1);
	\draw[lijn] (1,1) -- (2,1);
	\draw[lijn] (1,1) -- (1,0);
	\draw[lijn] (1,0) -- (2,0);
	\draw[lijn] (2,1) -- (2,0);
	\draw[lijn] (0,0) -- (-.5,-.5);
	\draw[lijn] (0,1) -- (-.5,1.5);
	\draw[lijn] (2,0) -- (2.5,-.5);
	\draw[lijn] (2,1) -- (2.5,1.5);
	%\draw (1,-1.5) node {$N_1$};
	\end{scope}

	% N_2
	 \begin{scope}[xshift=9cm,yshift=0cm]
	\draw[very thick, fill, radius=0.06] (-.5,1.5) circle node[left] {$a$};
	\draw[very thick, fill, radius=0.06] (-.5,-.5) circle node[left] {$b$};
	\draw[very thick, fill, radius=0.06] (2.5,1.5) circle node[right] {$d$};
	\draw[very thick, fill, radius=0.06] (2.5,-.5) circle node[right] {$c$};
	\draw[very thick, fill, radius=0.06] (0,0) circle;
	\draw[very thick, fill, radius=0.06] (1,0) circle;
	\draw[very thick, fill, radius=0.06] (0,1) circle;
	\draw[very thick, fill, radius=0.06] (1,1) circle;
	\draw[very thick, fill, radius=0.06] (2,1) circle;
	\draw[very thick, fill, radius=0.06] (2,0) circle;
	\draw[lijn] (0,0) -- (0,1);
	\draw[lijn] (0,0) -- (1,0);
	\draw[lijn] (0,1) -- (1,1);
	\draw[lijn] (1,0) -- (1,1);
	\draw[lijn] (1,1) -- (2,1);
	\draw[lijn] (1,1) -- (1,0);
	\draw[lijn] (1,0) -- (2,0);
	\draw[lijn] (2,1) -- (2,0);
	\draw[lijn] (0,0) -- (-.5,-.5);
	\draw[lijn] (0,1) -- (-.5,1.5);
	\draw[lijn] (2,0) -- (2.5,-.5);
	\draw[lijn] (2,1) -- (2.5,1.5);
	%\draw (1,-1.5) node {$N_2$};
	\end{scope}

	\end{tikzpicture}
\caption{\label{fig:notleafreconstructible} A pair of phylogenetic networks that are not leaf-reconstructible (and not even $V(N)$-reconstructible) but that are end-vertex reconstructible.}
\end{figure}

\begin{remark}
At first sight it might appear that leaf-reconstructibility of a phylogenetic network
could be equivalent to endvertex-reconstructibility (where one tries to reconstruct a graph from
the deck obtained by deleting only its endvertices, \leo{i.e. leaves}, cf. \cite[p.237]{bh77}). However,
these are distinct concepts. For example, the 
phylogenetic networks \leo{in Figure~\ref{fig:notendvertexreconstructible}}
are leaf-reconstructible. However, considered as graphs (with no labels), they are not 
endvertex-reconstructible, as they both have the same endvertex deck \cite[p.313]{kp81}. \leo{Conversely, the networks in Figure~\ref{fig:notleafreconstructible} are endvertex-reconstructible but not leaf-reconstructible.}
%It seems that binary leaf-reconstructible networks are always endvertex-reconstructible.
\end{remark}

%\centerline{\includegraphics[width=0.8\textwidth]{endvertex_reconstructibility.jpg}}

\begin{figure}
\centering
%\centerline{\includegraphics{fig/fig3}}
 \begin{tikzpicture}[scale=.8]
	 \tikzset{lijn/.style={very thick}}
	 \begin{scope}[xshift=0cm,yshift=0cm]
	\draw[very thick, fill, radius=0.06] (-.5,-.5) circle node[left] {$x$};
	\draw[very thick, fill, radius=0.06] (4.5,-.5) circle node[right] {$z$};
	\draw[very thick, fill, radius=0.06] (4.5,.5) circle node[right] {$y$};
	\draw[very thick, fill, radius=0.06] (0,0) circle;
	\draw[very thick, fill, radius=0.06] (2,0) circle;
	\draw[very thick, fill, radius=0.06] (4,0) circle;
	\draw[very thick, fill, radius=0.06] (2,.5) circle;
	\draw[very thick, fill, radius=0.06] (2,1) circle;
	\draw[very thick, fill, radius=0.06] (1.67,1.5) circle;
	\draw[very thick, fill, radius=0.06] (2.33,1.5) circle;
	\draw[very thick, fill, radius=0.06] (1.33,2) circle;
	\draw[very thick, fill, radius=0.06] (2.66,2) circle;
	\draw[very thick, fill, radius=0.06] (2,3) circle;
	\draw[lijn] (0,0) -- (4,0);
	\draw[lijn] (0,0) -- (2,.5);
	\draw[lijn] (2,0) -- (2,1);
	\draw[lijn] (4,0) -- (2.33,1.5);
	\draw[lijn] (4,0) -- (2,3);
	\draw[lijn] (0,0) -- (2,3);
	\draw[lijn] (2,1) -- (2.67,2);
	\draw[lijn] (2,1) -- (1.33,2);
	\draw[lijn] (2,3) -- (1.67,1.5);
	\draw[lijn] (-.5,-.5) -- (0,0);
	\draw[lijn] (4,0) -- (4.5,0.5);
	\draw[lijn] (4,0) -- (4.5,-0.5);
	\end{scope}
	
	 \begin{scope}[xshift=9cm,yshift=0cm]
	\draw[very thick, fill, radius=0.06] (-.5,-.5) circle node[left] {$z$};
	\draw[very thick, fill, radius=0.06] (4.5,-.5) circle node[right] {$x$};
	\draw[very thick, fill, radius=0.06] (-.5,.5) circle node[left] {$y$};
	\draw[very thick, fill, radius=0.06] (0,0) circle;
	\draw[very thick, fill, radius=0.06] (2,0) circle;
	\draw[very thick, fill, radius=0.06] (4,0) circle;
	\draw[very thick, fill, radius=0.06] (2,.5) circle;
	\draw[very thick, fill, radius=0.06] (2,1) circle;
	\draw[very thick, fill, radius=0.06] (1.67,1.5) circle;
	\draw[very thick, fill, radius=0.06] (2.33,1.5) circle;
	\draw[very thick, fill, radius=0.06] (1.33,2) circle;
	\draw[very thick, fill, radius=0.06] (2.66,2) circle;
	\draw[very thick, fill, radius=0.06] (2,3) circle;
	\draw[lijn] (0,0) -- (4,0);
	\draw[lijn] (0,0) -- (2,.5);
	\draw[lijn] (2,0) -- (2,1);
	\draw[lijn] (4,0) -- (2.33,1.5);
	\draw[lijn] (4,0) -- (2,3);
	\draw[lijn] (0,0) -- (2,3);
	\draw[lijn] (2,1) -- (2.67,2);
	\draw[lijn] (2,1) -- (1.33,2);
	\draw[lijn] (2,3) -- (1.67,1.5);
	\draw[lijn] (-.5,-.5) -- (0,0);
	\draw[lijn] (0,0) -- (-.5,0.5);
	\draw[lijn] (4,0) -- (4.5,-0.5);
	\end{scope}

	\end{tikzpicture}
\caption{\label{fig:notendvertexreconstructible} A pair of phylogenetic networks that are not end-vertex reconstructible but that are leaf-reconstructible.}
\end{figure}
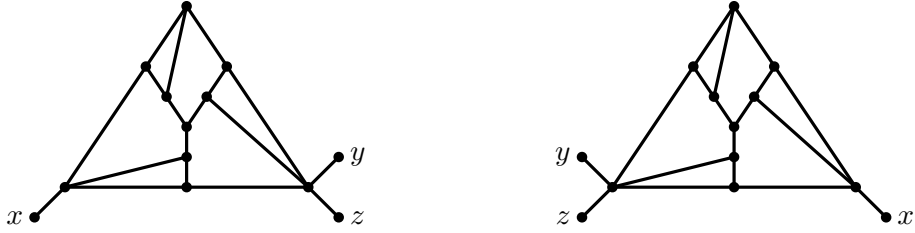

%We are usually interested in two cases: $U=X$ and~$U=V(N)$. Note that $X$-reconstructibility implies $V(N)$-reconstructibility, but not conversely (see Section~\ref{sec:trees} for an example).

We call a class~$\cN$ of phylogenetic networks \emph{leaf-reconstructible} if each~$N\in\cN$ is leaf-reconstructible. A function~$f$ defined on~$\cN$ is \emph{leaf-reconstructible} if for each~$N\in\cN$ and for any leaf-reconstrution~$M$ of~$N$ we have~$f(N)=f(M)$. Class~$\cN$ is \emph{weakly leaf-reconstructible} if, for each network~$N\in\cN$, all leaf-reconstructions of~$N$ that are in~$\cN$ are equivalent to~$N$. Class~$\cN$ is \emph{leaf-recognizable} if, for each network~$N\in\cN$, every leaf-reconstruction of~$N$ is also in~$\cN$.

%Similarly, we call a class~$\cN$ of phylogenetic networks \emph{vertex-reconstructible} if each~$N\in\cN$ is $V(N)$-reconstructible, \emph{weakly vertex-reconstructible} if, for each network~$N\in\cN$, all $V(N)$-reconstructions of~$N$ that are in~$\cN$ are equivalent to~$N$ and \emph{vertex-recognizable} if, for each network~$N\in\cN$, every $V(N)$-reconstruction of~$N$ is also in~$\cN$.
%
%\begin{observation}
%If a class~$\cN$ of phylogenetic networks is (weakly) vertex-reconstructible then it is (weakly) leaf-reconstructible.
%\end{observation}

\begin{observation}
A class~$\cN$ of phylogenetic networks is leaf-reconstructible if and only if it is leaf-recognizable and weakly leaf-reconstructible.
\end{observation}

We conclude this section by showing that certain features of a binary phylogenetic 
network on $X$ can be reconstructed from its $X$-deck.
The \emph{reticulation number} of a \leo{pseudo-}network~$N$ is defined as $|E(N)| - |V(N)| + 1$. The \emph{level} of~$N$ is the maximum reticulation number of a biconnected component of~$N$. \leo{A phylogenetic network} is called a \emph{level-$k$} network, with~$k\in\N$, if its level is at most~$k$.

\begin{proposition}
The number of edges, the number of vertices, the reticulation number and the level of a binary phylogenetic network~$N$ are leaf-reconstructible.
\end{proposition}
\begin{proof}
Let~$N$ be any phylogenetic network and~$x\in L(N)$.

If~$|V(N)|=2$, then $|V(N_x)| = |V(N)|-1$ and $|E(N_x)| = |E(N)|-1$. Moreover, the level and reticulation number of~$N_x$ are~0, the same as the reticulation number and level of~$N$. 

If~$|V(N)|\geq 3$, then $|V(N_x)| = |V(N)|-2$ and $|E(N_x)| = |E(N)|-2$. Moreover, the level and reticulation number of~$N_x$ are the same as the reticulation number and, respectively, level of~$N$. 

In both cases, the proposition follows directly. 
\end{proof}

The following is a direct consequence.

\begin{corollary}\label{cor:recog}
For each~$k\in\N$, the class of binary level-$k$ phylogenetic networks is leaf-recognizable.
\end{corollary}

\section{Decomposable networks}\label{sec:decomp}

In this section we will consider decomposable networks, that is, 
networks with at least one nontrivial cut-edge (that is, a cut-edge which
does not contain a leaf). We start with a few simple observations. Note that, for~$|X| \leq 3$, there exists a unique phylogenetic tree on~$X$ which is therefore $X$-reconstructible. For~$|X|=4$, no binary phylogenetic tree on~$X$ is $X$-reconstructible, but all phylogenetic trees~$T$ on~$X$ are $V(T)$-reconstructible.

\begin{theorem}\label{thm:trees}
\leo{Any phylogenetic tree with at least five leaves is leaf-}reconstructible.
\end{theorem}
\begin{proof}
The class of phylogenetic trees is \leo{leaf-}recognizable by Corollary~\ref{cor:recog}. To show weak-reconstructibility, suppose \leo{that there exist phylogenetic trees}~$T\not\sim T'$ \leo{on~$X$} such that~$T$ and~$T'$ have the same $X$-deck. Then there is at least one nontrivial split $A|B$ that is a split of, without loss of generality,~$T$ but not of~$T'$. Since~$|X|\geq 5$, at least one of~$A$ and~$B$ contains at least three elements. The other side contains at least two elements since the split is nontrivial. Assume $a_1,a_2,a_3\in A$ and~$b_1,b_2\in B$. Then~$T_{a_1}$ has split $A\setminus \{a_1\}|B$ and \leo{$T_{a_2}$ has split} $A\setminus \{a_2\}|B$. Hence, $T'_{a_1}$ and $T'_{a_2}$ have the same splits\leo{, respectively}. This implies that~$T'$ has a split that can be obtained from $A\setminus \{a_1\}|B$ by inserting~$a_1$. Since it does not have split~$A|B$, it must have split $A\setminus \{a_1\}|B\cup\{a_1\}$. Similarly,~$T'$ must have the split $A\setminus \{a_2\}|B\cup\{a_2\}$. This leads to a contradiction because these splits are incompatible \leo{(see~e.g.~\cite{SS})}. 
\end{proof}

\begin{remark}
It is known that any tree is reconstructible \cite{kelly57}. A proof of this result is given in \cite[p.232]{bh77},
which uses a generalization of Kelly's Lemma \cite{kelly57}. Kelly's Lemma is key 
to proving several results in graph reconstructibility. We were unable to derive an analogous 
result for leaf-reconstructibility --  it would be interesting to know if some such result exists.
Note also that trees are known to be endvertex\leo{-}reconstructible \cite{HP66}. 
\end{remark}

\leo{To extend Theorem~\ref{thm:trees} to decomposable networks, we will use the following observation.}

\begin{observation}\label{obs:treedeck}
For any phylogenetic network~$N$ on~$X$ and any leaf~$x\in X$ we have
\[
(T(N))_x = T(N_x)
\]
\end{observation}

\begin{corollary}\label{cor:TN}
The function mapping a phylogenetic network~$N$ with at least five leaves to~$T(N)$ is leaf-reconstructible.
\end{corollary}
\begin{proof}
By Observation~\ref{obs:treedeck} and Theorem~\ref{thm:trees}.
\end{proof}

\begin{theorem}\label{thm:cutedge}
Any decomposable phylogenetic network with at least five leaves is leaf-reconstructible.
\end{theorem}
\begin{proof} Let~$\cN$ be the class of phylogenetic networks with at least five leaves and at least one nontrivial cut-edge. This class is leaf-recognizable since a phylogenetic network on~$X$ belongs to this class if and only if every element of its $X$-deck has four leaves and at most two elements of its \leo{$X$-}deck have no nontrivial cut-edges.

It remains to show weak leaf-reconstructibility. Suppose~$|X|\geq 5$ and let~$N$ be a phylogenetic network on~$X$ with some nontrivial cut-edge~$e$. Let~$A|B$ be the split induced by~$e$. By Corollary~\ref{cor:TN}, $T(N)$ is $X$-reconstructible. Hence, any reconstruction~$N'$ of~$N$ contains a unique edge~$e'$ representing split~$A|B$. Since~$e$ is nontrivial, there exist leaves~$a_1,a_2\in A$ and~$b_1,b_2\in B$. \leo{Pseudo-}network~$N_{a_1}$ contains a \leo{unique} edge~$f$ inducing split $A\setminus\{a_1\}|B$. \leo{Since $N_{a_1}\sim N'_{a_1}$}, the connected component of $N_{a_1} - f$ containing~$B$ is \leo{equivalent} to the connected component of $N'-e'$ containing~$B$. Call this connected component~$N_B$ and let~$u$ be the endpoint of~$f$ that it contains. Similarly, \leo{pseudo-}network~$N_{b_1}$ contains a \leo{unique} edge~$g$ inducing split $A|B\setminus\{b_1\}$ and the connected component of~$N_{b_1}-g$ containing~$A$ is \leo{equivalent} to the connected component of~$N'-e'$ containing~$A$. Call this connected component~$N_A$ and let~$v$ be the endoint of~$g$ that it contains. Then,~$N'$ can be obtained from~$N_A$ and~$N_B$ by adding an edge between~$u$ and~$v$. Therefore,~$N'\sim N$.
\end{proof}

%\section{Level-1 networks}\label{sec:lev1}

\section{Simple networks}\label{sec:simple}

When considering leaf-reconstructability of binary networks we can, by Theorem~\ref{thm:cutedge}, restrict to simple networks, which are binary networks containing precisely one blob.
Therefore, in this section we focus on leaf-reconstructibility of simple binary networks. \leo{The class of such networks is clearly leaf-recognizable since a phylogenetic network on~$X$ is contained in this class if and only if each element of its~$X$-deck is binary and has precisely one blob.}

We say that $(x,y,z)$ is a \emph{3-chain} of a phylogenetic network~$N$ on~$X$ if~$x,y,z\in X$ and~$N$ contains a path~$(u,v,w)$ such that~$x,y$ and~$z$ are respectively a neighbour of~$u,v$ and~$w$.

% Define a \emph{3-chain} as path of length~3 such that each vertex of the path has a leaf as neighbour.

\begin{lemma}\label{lem:3chain}
Any simple binary level-$k$ phylogenetic network \leo{containing a 3-chain is leaf-reconstructible if it has at least~4 leaves and at least~5 leaves if~$k=1$}.
\end{lemma}
\begin{proof}
\leo{The class~$\cN$ of such networks is leaf-recognizable since a simple binary level-$k$ phylogenetic network \leo{on~$X$}, with~$|X|\geq 4$ and~$|X|\geq 5$ if~$k=1$, is contained in~$\cN$ if and only if at most three elements of its $X$-deck do not contain a 3-chain.}

To show weak leaf-reconstructibility, let~$N\in\cN$ \leo{be a phylogenetic network on~$X$} and let~$(x,y,z)$ be a 3-chain in~$N$. Since~$|X|\geq 4$, there exists at least one other leaf~$a\in X$. Consider~$N_y$ and~$N_a$. First observe that~$N_a$ contains a 3-chain~$(x,y,z)$.
%Moreover, in~$N_a$ there is no edge between the neighbours of~$x$ and~$z$ by the assumption that~$|X|\geq 5$ if~$k=1$.
In~$N_y$, there is a unique edge~$e$ between the neighbours of~$x$ and~$z$. Moreover, in~$N_y$ there is no 3-chain $(x,a,z)$ by the assumption that~$|X|\geq 5$ if~$k=1$.
Let~$N'\in\cN$ be a $\{y,a\}$-reconstruction of~$N$. Then~$N'$ contains a 3-chain $(x,y,z)$ since $N_a$ contains a 3-chain~$(x,y,z)$ and~$N_y$ does not contain a 3-chain $(x,a,z)$. Hence, $N'$ can be reconstructed from~$N_y$ by \leo{attaching~$y$ to edge~$e$}. Therefore,~$N'\sim N$.
\end{proof}

\begin{corollary}\label{cor:levk}
Any simple binary level-$k$ phylogenetic network \leo{with at least $6k-5$ leaves and~$k\geq 2$ is leaf-reconstructible.}
\end{corollary}
\begin{proof}
\leo{Leaf-recognizability is clear.} Let~$N$ be a simple binary level-$k$ phylogenetic network on~$X$ with~$k\geq 2 $ and~$|X|\geq 6k-5$. Deleting all leaves from~$N$ and suppressing all degree-2 vertices gives a 3-regular multigraph~$G$. Since~$N$ is simple level-$k$, $|E(N)|-|V(N)|+1=k$ and hence $|E(G)|-|V(G)|+1=k$. Combining this with the fact that, since~$G$ is 3-regular, $3|V(G)|=2|E(G)|$ gives that $|E(G)|=3k-3$. Suppose that~$N$ contains no 3-chain. Then it could have at most two leaves per edge of~$G$, implying that $|X|\leq 6k-6$. Hence, $N$ contains a 3-chain and is therefore $X$-reconstructible by Lemma~\ref{lem:3chain}. 
\end{proof}

\begin{corollary}\label{cor:manyleaves}
Any binary phylogenetic network~$N=(V,E)$ on~$X$ with~$|X|\geq \min\{6(|E|-|V|)+1,5\}$ is \leo{leaf-}reconstructible.
\end{corollary}
\begin{proof}
If~$N$ contains a nontrivial cut-edge, then apply Theorem~\ref{thm:cutedge}. If it is simple level-$1$, then apply Lemma~\ref{lem:3chain}. If it is simple level-$k$ with~$k\geq 2$ then $|E|-|V|+1=k$ and hence $|X|\geq 6k-5$ and therefore we can apply Corollary~\ref{cor:levk}.
\end{proof}

We say that \emph{almost all} phylogenetic networks from a certain class~$\cN$ are leaf-recon\-structible, if the probability that a network drawn uniformly at random out of all networks in~$\cN$ with~$n$ leaves is leaf-reconstructible goes to~1 when~$n$ goes to infinity.

\begin{corollary}\label{cor:almostall}
For any fixed~$k$, almost all \leo{binary} level-$k$ phylogenetic networks are leaf-reconstructible.
\end{corollary}
\begin{proof}
All networks with at least five leaves and some nontrivial cut-edge are leaf-reconstructible by Theorem~\ref{thm:cutedge}. For a simple \leo{binary} level-$k$ phylogenetic network~$N=(V,E)$ on~$X$, with~$k\geq 1$ we have (similar to in the proof of Corollary~\ref{cor:levk})
\[
|V| =  2k-2 + 2|X|.
\]
Hence, when $|V|\rightarrow \infty$ then $|X|\rightarrow \infty$. When $|X|\geq \min\{6k-5,5\}$ then~$N$ is $X$-reconstructible by Lemma~\ref{lem:3chain} and Corollary~\ref{cor:levk}. The corollary follows.
\end{proof}

\section{Reconstruction numbers of decomposable networks}\label{sec:number}

In this section, we shall show that \leo{the} reconstruction number of a decomposable phylogenetic network
with at least five leaves is \leo{at most} two.

\begin{observation}\label{obs:levnum}
\leo{Let~$k\geq 0$.} To recognize that a phylogenetic network~$N$ is level-$k$ it suffices to check that any element of its~$X$-deck is level-$k$.
\end{observation}

We start by determining the reconstruction number of binary trees.

The \emph{median} of three leaves~$x,y,z\in L(T)$ in a phylogenetic tree~$T$ is the unique vertex that lies on each of the paths between \leo{all} pairs of leaves in~$\{x,y,z\}$.

\begin{lemma}\label{lem:recnumbintrees}
Any binary phylogenetic tree~$T$ \leo{with at least five leaves} has \leo{leaf-}recon\-struction number~2.
\end{lemma}
\begin{proof}
The class of phylogenetic trees on~$X$ is $\{x\}$-recognizable for any~$x\in X$ by Observation~\ref{obs:levnum}. \leo{No phylogenetic tree on~$X$ with $|X|\geq 5$ is~$\{x\}$-reconstructible for any~$x\in X$ since attaching~$x$ to different edges in~$T_x$ gives different non-equivalent trees. Hence, the leaf-reconstruction number of such trees is at least~2. It remains to show that it is exactly~2.}

\leo{Consider a binary phylogenetic tree~$T$ on~$X$ with~$|X|\geq 5$.} Take any two leaves~$x,y\in X$ such that the distance between them is at least~4. Such leaves exist since~$|X|\geq 5$. We will show that~$T$ can be uniquely reconstructed from~$T_x$ and~$T_y$. First observe that any \leo{leaf}-reconstruction of~$T$ is binary since~$T_x$ and~$T_y$ are binary and~$x$ and~$y$ do not have a common neighbour.

Let~$w$ be the neighbour of~$x$ in~$T$ and~$u,v$ the other two neighbours of~$w$. Then~$T_x$ has an edge~$\{u,v\}$.

First assume that neither~$u$ nor~$v$ is a leaf. Then there exist leaves~$a,b\neq y$ such that the path between~$a$ and~$b$ (in~$T$) contains~$u$ but not~$w$ and there exist leaves~$c,d\neq y$ such the path between~$c$ and~$d$ (in~$T$) contains~$v$ but not~$w$. Then~$u$ is the median of~$a,b,c$ and~$v$ is the median of~$a,c,d$ in~$T$. \leo{Call in~$T_x$ and~$T_y$ the median of~$a,b,c$ also~$u$ and the median of $a,c,d$ also~$v$. Then, in}~$T_y$, the neighbour of~$x$ is adjacent to~$u$ and~$v$. Hence, we can reconstruct~$T$ from~$T_x$ by \leo{attaching~$x$ to} the edge~$\{u,v\}$.

Now assume that~$u$ is a leaf. Then there again exist leaves~$c,d\neq y$ such that~$v$ is on the path between~$c$ and~$d$ (in~$\leo{T}$). \leo{In this case,}~$v$ is the median of~$u,c,d$ in~$T$. \leo{Call the median of~$u,c,d$ in~$T_x$ and~$T_y$ also~$v$. Then, since the neighbour of~$x$ in~$T_y$ is adjacent to~$u$ and~$v$, we} can again uniquely reconstruct~$T$ from~$T_x$ by \leo{attaching~$x$ to} the edge~$\{u,v\}$.
\end{proof}

We now consider nonbinary trees.

\begin{theorem}\label{thm:recnumtrees}
Any phylogenetic tree \leo{with at least five leaves has leaf-}reconstruction number~2 unless it is a star, in which case it has \leo{leaf-}reconstruction number~3.
\end{theorem}
\begin{proof}
\leo{As in the proof of Lemma~\ref{lem:recnumbintrees}, it is clear that, for any~$x\in X$, the class of phylogenetic trees on~$X$ is~$\{x\}$-recognizable and no phylogenetic tree on~$X$ is $\{x\}$-reconstructible if~$|X|\geq 5$. Consider a phylogenetic tree~$T$ on~$X$ with~$|X|\geq 5$.}

First consider the case that~$T$ is a star. Then, for any~$x,y\in X$, there exists a phylogenetic tree~$T'\not\sim T$ on~$X$ such that~$T'_x\sim T_x$ and~$T'_y\sim T_y$. Hence, the $X$-reconstruction number of~$T$ is at least~3. To see that it is exactly~3, note that any phylogenetic tree that is not a star has at most two \leo{elements in its $X$-deck that are stars}. Hence, since there exists a unique phylogenetic star tree on~$X$, the reconstruction number of~$T$ is~3.

Now consider the case that~$T$ contains exactly one nontrivial cut-edge~$\{u,v\}$. Take one leaf~$x$ adjacent to~$u$ and one leaf~$y$ adjacent to~$v$. First suppose that~$u$ has degree~\leo{3}. Then~$v$ has degree at least~\leo{4}. Hence,~$T_x$ is a star tree and~$T_y$ has exactly one nontrivial cut-edge~$\{u',v'\}$. Suppose~$x$ is adjacent to~$u'$. Then~$u'$ is adjacent to exactly one other leaf~$z$. Hence, \leo{we can uniquely reconstruct~$T$ from~$T_x$ by attaching~$x$ to the edge incident to~$z$.} Now suppose that both~$u$ and~$v$ have degree at least~3. Then~$T_x$ and~$T_y$ both have exactly one nontrivial cut-edge. Let~$z$ be any leaf adjacent to the neighbour of~$x$ in~$T_y$. Then \leo{we can uniquely reconstruct~$T$ from~$T_x$ by adding~$x$ with an edge to the neighbour of~$z$.}

Finally, assume that~$T$ has at least two nontrivial cut-edges. Then there exist two leaves~$x,y\in X$ such that the distance between them is at least~4. Let~$w$ be the neighbour of~$x$ in~$T$ and~$u,v\neq x$ two other neighbours of~$w$.

If~$w$ has degree~3, then we can proceed as in the proof of Lemma~\ref{lem:recnumbintrees}. 

Now assume~$w$ has degree at least~4. Then it has a neighbour~$z\notin\{u,v,x\}$. Then there exist leaves~$a,b,c\notin\{x,y\}$ reachable by paths from~$u,v$ and~$z$ respectively that do not contain~$w$. Therefore, the median of~$a,b$ and~$c$ \leo{in~$T$ is~$w$. Hence, we can uniquely reconstruct~$T$ from~$T_x$ by adding~$x$ with an edge to} the median of~$a,b$ and~$c$.
\end{proof}

\begin{corollary}\label{cor:cutedge}
Any decomposable phylogenetic network with at least five leaves has leaf-reconstruction number \leo{at most}~2.
\end{corollary}
\begin{proof}
Let~$N$ be a phylogenetic network that has at least five leaves and at least one nontrivial cut-edge and let~$x$ and~$y$ be maximum distance apart in~$T(N)$. Then any $\{x,y\}$-reconstruction has a nontrivial cut-edge. Moreover, since the distance between~$x$ and~$y$ in~$T(N)$ is at least~3, $T(N)$ is $\{x,y\}$-reconstructable by the proof of Theorem~\ref{thm:recnumtrees}. Moreover, by the proof of Theorem~\ref{thm:cutedge}, it now follows that~$N$ is $\{x,y\}$-reconstructable. 
\end{proof}

\section{Low-level networks}\label{sec:low}

In this section we show that all \leo{binary} networks with \leo{at least five leaves and} level at most~4 are leaf-reconstructible and, moreover, have leaf-reconstruction number at most~2. The proofs are based on the following notions.

\begin{definition} A binary level-$k$ \emph{generator}, for~$k\geq 2$, is a 2-connected 3-regular multigraph~$G=(V,E)$ with~$|E|-|V|+1=k$. The \emph{underlying generator} of a binary simple level-$k$ network~$N$ is the generator obtained from~$N$ by deleting all leaves and suppressing resulting degree-2 vertices. \leo{For an edge~$e$ of~$G$,} we say that a leaf~$x$ \emph{is on edge}~$e$ in~$N$ if the neighbour of~$x$ is on a path that is suppressed into edge~$e$. If~$x$ is on edge~$e$ then we also say that~$e$ \emph{contains}~$x$ and we refer to~$e$ as the \emph{$x$-edge}.
\end{definition}

See Figure~\ref{fig:generators} for all binary level-$k$ generators, for $2\leq k\leq 4$.

\begin{figure}
\centering
%\centerline{\includegraphics{fig/fig4}}
 \begin{tikzpicture}
	 \tikzset{lijn/.style={very thick}}
	 
	 %\draw[step=1cm,gray,very thin] (0,0) grid (12,-12);
	 
	% level-2
	\draw (-1,0) node {level-2};
	\draw[very thick, fill, radius=0.06] (1,0) circle;
	\draw[very thick, fill, radius=0.06] (2,0) circle;
	\draw[lijn] (1,0) -- (2,0);
	\draw[lijn] (1,0) to[out=55,in=125] (2,0);
	\draw[lijn] (1,0) to[out=-55,in=-125] (2,0);	
	
	% level-3
	\draw (-1,-1.5) node {level-3};
	% cylinder
	\draw[very thick, fill, radius=0.06] (1,-2) circle;
	\draw[very thick, fill, radius=0.06] (2,-2) circle;
	\draw[lijn] (1,-2) to[out=45,in=135] (2,-2);
	\draw[lijn] (1,-2) to[out=-45,in=-135] (2,-2);	
	\draw[very thick, fill, radius=0.06] (1,-1) circle;
	\draw[very thick, fill, radius=0.06] (2,-1) circle;
	\draw[lijn] (1,-1) to[out=45,in=135] (2,-1);
	\draw[lijn] (1,-1) to[out=-45,in=-135] (2,-1);	
	\draw[lijn] (1,-2) -- (1,-1);
	\draw[lijn] (2,-2) -- (2,-1);
	% K_4
	\begin{scope}[xshift=0.0cm,yshift=-0.2cm]
	\draw[very thick, fill, radius=0.06] (5,-1.5) circle;
	\draw[very thick, fill, radius=0.06] (5,-.6) circle;
	\draw[very thick, fill, radius=0.06] (4.3,-2) circle;
	\draw[very thick, fill, radius=0.06] (5.7,-2) circle;
	\draw[lijn] (5,-1.5) -- (5,-.6);
	\draw[lijn] (5,-1.5) -- (4.3,-2);
	\draw[lijn] (5,-1.5) -- (5.7,-2);
	\draw[lijn] (5,-.6) -- (4.3,-2);
	\draw[lijn] (5,-.6) -- (5.7,-2);
	\draw[lijn] (4.3,-2) -- (5.7,-2);
	\end{scope}
	
	% level-4
	\draw (-1,-4.5) node {level-4};
	% G1
	\begin{scope}[xshift=0.0cm,yshift=-0.5cm]
	\draw[very thick, fill, radius=0.06] (1,-5) circle;
	\draw[very thick, fill, radius=0.06] (2,-5) circle;
	\draw[lijn] (1,-5) to[out=45,in=135] (2,-5);
	\draw[lijn] (1,-5) to[out=-45,in=-135] (2,-5);	
	\draw[very thick, fill, radius=0.06] (1,-4) circle;
	\draw[very thick, fill, radius=0.06] (2,-4) circle;
	\draw[lijn] (1,-4) -- (1,-5);
	\draw[lijn] (2,-4) -- (2,-5);	
	\draw[lijn] (1,-4) -- (2,-4);	
	\draw[very thick, fill, radius=0.06] (1,-3) circle;
	\draw[very thick, fill, radius=0.06] (2,-3) circle;
	\draw[lijn] (1,-3) to[out=45,in=135] (2,-3);
	\draw[lijn] (1,-3) to[out=-45,in=-135] (2,-3);	
	\draw[lijn] (1,-4) -- (1,-3);
	\draw[lijn] (2,-4) -- (2,-3);
	\draw (1.5,-5.7) node {$G_1$};
	\end{scope}
	
	% old G2
%	\draw[very thick, fill, radius=0.06] (4,-5) circle;
%	\draw[very thick, fill, radius=0.06] (6,-5) circle;
%	\draw[lijn] (4,-5) -- (6,-5);	
%	\draw[very thick, fill, radius=0.06] (5,-4.5) circle;
%	\draw[lijn] (4,-5) -- (5,-4.5);	
%	\draw[lijn] (6,-5) -- (5,-4.5);	
%	%\draw[very thick, fill, radius=0.06] (4,-4) circle;
%	\draw[very thick, fill, radius=0.06] (6,-4) circle;
%	\draw[lijn] (4,-4) -- (4,-5);
%	\draw[lijn] (6,-4) -- (6,-5);	
%	\draw[lijn] (5,-4.5) -- (6,-4);	
%	\draw[very thick, fill, radius=0.06] (4,-3) circle;
%	\draw[very thick, fill, radius=0.06] (6,-3) circle;
%	\draw[lijn] (4,-3) to[out=35,in=145] (6,-3);
%	\draw[lijn] (4,-3) to[out=-35,in=-145] (6,-3);	
%	\draw[lijn] (4,-4) -- (4,-3);
%	\draw[lijn] (6,-4) -- (6,-3);	
%	\draw (5,-5.5) node {$G_2$};
	
	% G3
	\begin{scope}[xshift=0.0cm,yshift=-0.5cm]
	\draw[very thick, fill, radius=0.06] (8,-5) circle;
	\draw[very thick, fill, radius=0.06] (9,-5) circle;
	\draw[lijn] (8,-5) to[out=45,in=135] (9,-5);
	\draw[lijn] (8,-5) to[out=-45,in=-135] (9,-5);	
	\draw[very thick, fill, radius=0.06] (9,-3.5) circle;
	\draw[very thick, fill, radius=0.06] (9,-4.5) circle;
	\draw[lijn] (8,-4) -- (8,-5);
	\draw[lijn] (9,-4.5) -- (9,-5);	
	\draw[lijn] (9,-3.5) to[out=-45,in=45] (9,-4.5);	
	\draw[lijn] (9,-3.5) to[out=-135,in=135] (9,-4.5);	
	\draw[very thick, fill, radius=0.06] (8,-3) circle;
	\draw[very thick, fill, radius=0.06] (9,-3) circle;
	\draw[lijn] (8,-3) to[out=45,in=135] (9,-3);
	\draw[lijn] (8,-3) to[out=-45,in=-135] (9,-3);	
	\draw[lijn] (8,-4) -- (8,-3);
	\draw[lijn] (9,-3.5) -- (9,-3);
	\draw (8.5,-5.7) node {$G_3$};
	\end{scope}
	
	% G2 (former G4)
	\begin{scope}[xshift=3.5cm,yshift=3cm]
	\draw[very thick, fill, radius=0.06] (.5,-8.5) circle;
	\draw[very thick, fill, radius=0.06] (2.5,-8.5) circle;
	\draw[lijn] (.5,-8.5) to[out=35,in=145] (2.5,-8.5);
	\draw[lijn] (.5,-8.5) to[out=-35,in=-145] (2.5,-8.5);	
	\draw[very thick, fill, radius=0.06] (.5,-6.5) circle;
	\draw[very thick, fill, radius=0.06] (2.5,-6.5) circle;
	\draw[lijn] (.5,-6.5) -- (1.5,-6);
	\draw[lijn] (.5,-6.5) -- (1.5,-7);
	\draw[lijn] (1.5,-6) -- (1.5,-7);
	\draw[lijn] (2.5,-6.5) -- (1.5,-7);
	\draw[lijn] (2.5,-6.5) -- (1.5,-6);
	\draw[very thick, fill, radius=0.06] (1.5,-6) circle;
	\draw[very thick, fill, radius=0.06] (1.5,-7) circle;
	\draw[lijn] (.5,-8.5) -- (.5,-6.5);
	\draw[lijn] (2.5,-8.5) -- (2.5,-6.5);
	\draw (1.5,-9.2) node {$G_2$};
	\end{scope}
	
	% G4 (former G5)
	\begin{scope}[xshift=0.5cm,yshift=-9.7cm]
	\draw[very thick, fill, radius=0.06] (0,0) circle;
	\draw[very thick, fill, radius=0.06] (2,0) circle;
	\draw[very thick, fill, radius=0.06] (1,1) circle;
	\draw[very thick, fill, radius=0.06] (0,3) circle;
	\draw[very thick, fill, radius=0.06] (2,3) circle;
	\draw[very thick, fill, radius=0.06] (1,2) circle;
	\draw[lijn] (0,0) -- (2,0);
	\draw[lijn] (0,0) -- (1,1);
	\draw[lijn] (2,0) -- (1,1);
	\draw[lijn] (0,3) -- (2,3);
	\draw[lijn] (0,3) -- (1,2);
	\draw[lijn] (2,3) -- (1,2);
	\draw[lijn] (0,0) -- (0,3);
	\draw[lijn] (2,0) -- (2,3);
	\draw[lijn] (1,1) -- (1,2);
	\draw (1,-0.5) node {$G_4$};
	\end{scope}

	% G5 (former G6)
	\begin{scope}[xshift=4cm,yshift=-9.7cm]
	\draw[very thick, fill, radius=0.06] (1,0.3) circle;
	\draw[very thick, fill, radius=0.06] (0,1) circle;
	\draw[very thick, fill, radius=0.06] (2,1) circle;
	\draw[very thick, fill, radius=0.06] (0,2) circle;
	\draw[very thick, fill, radius=0.06] (2,2) circle;
	\draw[very thick, fill, radius=0.06] (1,2.7) circle;
	\draw[lijn] (1,0.3) -- (0,1);
	\draw[lijn] (1,0.3) -- (2,1);
	\draw[lijn] (1,0.3) -- (1,2.7);
	\draw[lijn] (0,1) -- (2,2);
	\draw[lijn] (0,1) -- (0,2);
	\draw[lijn] (2,1) -- (0,2);
	\draw[lijn] (2,1) -- (2,2);
	\draw[lijn] (0,2) -- (1,2.7);
	\draw[lijn] (2,2) -- (1,2.7);
	\draw (1,-0.5) node {$G_5$};
	\end{scope}
	\end{tikzpicture}
\caption{\label{fig:generators} All binary level-$k$ generators, for $2\leq k\leq 4$.
}
\end{figure}
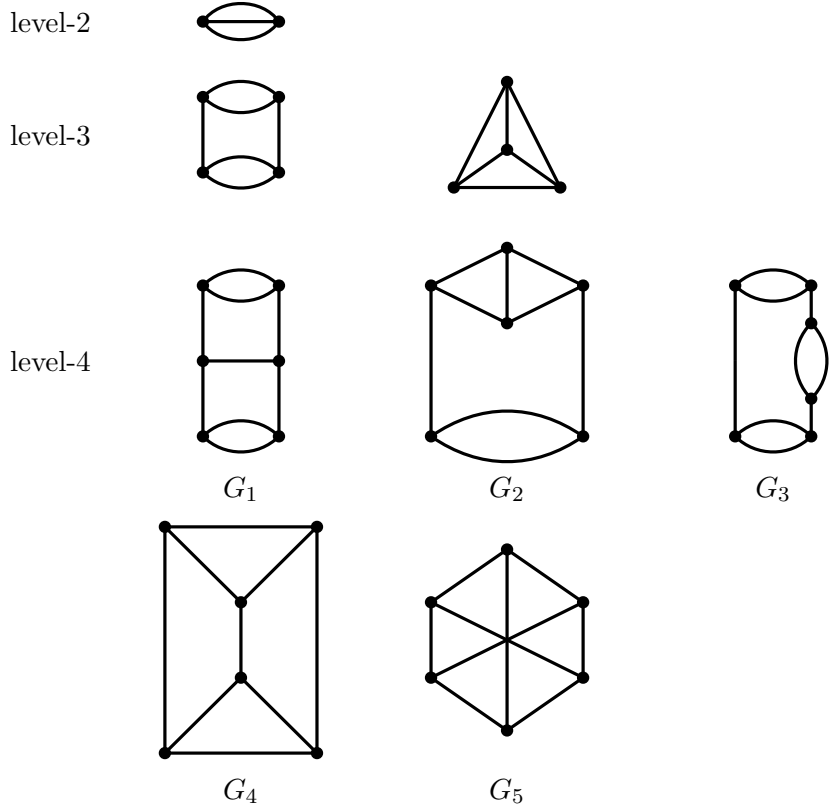

We say that two cycles are \emph{similar} if they have the same number of vertices and the same number of vertices that are neighbours of leaves, and hence also the same number of generator vertices (i.e. vertices that are not neighbours of leaves). 

% A $(p,\ell)$-\emph{cycle} is a cycle with~$p$ vertices of which~$\ell$ vertices have a neighbour that is a leaf.

\leo{The following three lemmas show several special cases of simple level-$k$ networks that are leaf-reconstructible. We will use these lemmas to show that all simple level-4 networks are leaf-reconstructible, if they have at least five leaves.}

\begin{lemma}\label{lem:4cycle}
Let~$N$ be a binary simple level-$k$ network on~$X$, with~$k\geq 2$ and~$|X|\geq 5$. If~$N$ contains a cycle~$C$ containing the neighbours of leaves~$a,b,c$ and~$d$ and either
\begin{itemize}
\item[(i)] there is no cycle~$C'\neq C$ in~$N$ that is similar to~$C$ and contains the neighbours of~$a,b$ and~$c$; or
\item[(ii)] $c$ and~$d$ are on the same edge of the underlying generator and there is no cycle~$C'\neq C$ in~$N$ that is similar to~$C$ and contains the neighbours of~$a,b,c$ and~$d$ in a different order,
\end{itemize}
then~$N$ is $\{d,e\}$-reconstructible, for any~$e\in X\setminus\{a,b,c,d\}$.
\end{lemma}
\begin{proof}
(i) Note that~$N_e$ has a cycle~\leo{$C_e$} containing the neighbours of~$a,b,c$ and~$d$ and no other cycle that is similar to~$C_e$ and contains the neighbours of~$a,b,c$ and~$d$. Assume without loss of generality that these neighbours are visited in this order. Suppose that the neighbour of~$d$ is the $i$-th vertex on the path from the neighbour of~$c$ to the neighbour of~$a$ on~$C_e$. Now consider~$N_d$, which contains a cycle~$C_d$ containing the neighbours of~$a,b$ and~$c$ and no other cycle similar to~$C_d$ that contains the neighbours of~$a$,~$b$ and~$c$. Let~$P$ be the path from the neighbour of~$c$ to the neighbour of~$a$ on~$C_d$\leo{, not via the neighbour of~$b$}. If the neighbour of~\leo{$e$} is among the first~$i$ vertices of~$P$ then we let~$f$ be the \leo{$i$-th} edge on~$P$. Otherwise, we let~$f$ be the \leo{$(i-1)$-th} edge on~$P$. Then the unique way to insert~$d$ into~$N_d$ is by attaching it to edge~$f$.

(ii) Assume without loss of generality that the distance between~$c$ and~$d$ is~3. Note that~$N_e$ has a cycle~\leo{$C_e$} containing the neighbours of~$a,b,c$ and~$d$ and no cycle that is similar to~$C_e$ and contains the neighbours of~$a,b,c$ and~$d$ in a different order. Assume again that~$C_e$ visits~$a,b,c$ and~$d$ in this order. Now consider~$N_d$ and choose any cycle~\leo{$C_d$} containing the neighbours of~$a,b$ and~$c$. Let~$f$ be the first edge on the path from the neighbour of~$c$ to the neighbour of~$a$ along~$C_d$, \leo{not via the neighbour of~$b$}. Then the unique way to insert~$d$ into~$N_d$ is by attaching it to edge~$f$.
\end{proof}

% WE DO NOT USE THIS YET
%\begin{lemma}\label{lem:3cycle}
%Let~$N$ be a binary simple level-$k$ network on~$X$, with~$k\geq 2$ and~$|X|\geq 5$. If~$N$ contains a cycle~$C$ containing the neighbours of leaves~$a,b$ and~$c$, and the path between~$a$ and~$b$ over~$C$ via~$c$ contains a different number of generator vertices than the path over~$C$ not containing~$c$, and there is no cycle~$C'\neq C$ that is similar to~$C$ and contains the neighbours of~$a$ and~$b$, then~$N$ is $\{c,d\}$-reconstructible, for any~$d\in X\setminus\{a,b,c\}$.
%\end{lemma}

%\begin{lemma}\label{lem:ordered4cycle}
%Let~$N$ be a binary simple level-$k$ network on~$X$, with~$k\geq 2$ and~$|X|\geq 5$. If~$N$ contains a cycle~$C$ containing the neighbours of leaves~$a,b,c$ and~$d$, where~$c$ and~$d$ are on the same side of the underlying generator, and there is no cycle~$C'\neq C$ that is similar to~$C$ and contains the neighbours of~$a,b,c$ and~$d$, in a different order, then~$N$ is $\{d,e\}$-reconstructible, for any~$e\in X\setminus\{a,b,c,d\}$.
%\end{lemma}

\begin{lemma}\label{lem:bubble}
Let~$N$ be a binary simple level-$k$ network on~$X$, with~$k\geq 2$ and~$|X|\geq 5$. If the underlying generator of~$N$ has a pair of multi-edges~$e_1,e_2$ then, unless one of~$e_1,e_2$ contains two leaves and the other one no leaves in~$N$, then~$N$ has leaf-reconstruction number at most~2.
\end{lemma}
\begin{proof}
First suppose that there is exactly one leaf~$x$ that is on one of the multi-edges. Then~$N_x$ has multi-edges. Since multi-edges are not allowed in phylogenetic networks, the unique way to insert~$x$ into~$N_x$ is by attaching it to one of the multi-edges.

Now suppose that there is exactly one leaf~$x$ on~$e_1$ and exactly one leaf~$a$ on~$e_2$. Let~$y$ be any other leaf. Then~$N_y$ contains a \leo{unique} 4-cycle containing the neighbours of~$x$ and~$a$, \leo{and these neighbours are not adjacent.} Since~$N_x$ contains a unique 3-cycle~$C$ containing the neighbour of~$a$, the only way to insert~$x$ into~$N_x$ is by attaching it to the unique edge on~$C$ that is not incident to the neighbour of~$a$.

Now suppose that there are exactly two leaves~$a,b$ on~$e_1$ and exactly one leaf~$x$ on~$e_2$. Let~$y\in X\setminus\{a,b,x\}$. Then, $N_y$ contains a unique 5-cycle containing the neighbours of~$a,b$ and~$x$ \leo{and the neighbour of~$x$ is not adjacent to the neighbours of~$a$ and~$b$.} Since~$N_x$ contains a unique 4-cycle~$C$ containing the neighbours of~$a$ and~$b$, the unique way to insert~$x$ into~$N_x$ is by attaching it to the unique edge on~$C$ that is not incident to the neighbours \leo{of~$a$ and~$b$}.

Now suppose that there are exactly two leaves~\leo{$a,b$} on~$e_1$ and exactly two leaves~\leo{$c,d$} on~$e_2$. This case is handled by Lemma~\ref{lem:4cycle}~(i).

The only remaining possibility is that there is a 3-chain, which is handled by the proof of Lemma~\ref{lem:3chain}.
\end{proof}

\begin{lemma}\label{lem:3incident}
Let~$N$ be a binary simple level-$k$ network on~$X$, with~$k\geq 2$ and~$|X|\geq 5$. If the underlying generator of~$N$ has three pairwise incident edges and~$N$ has at least three leaves on these edges, then~$N$ has leaf-reconstruction number at most~2.
\end{lemma}
\begin{proof}
First suppose that all three edges are incident to some vertex~$v$ and the other three endpoints are all distinct. If each edge contains at least one leaf, let~$a,b,c$ be the leaves closest to~$v$ on each of the edges. Then~$N$ is~$\{a,d\}$-reconstructible for any~$d\in X\setminus \{a,b,c\}$, since we can reconstruct~$N$ from~$N_a$ by attaching~$a$ to the edge that is incident to \leo{the vertex~$v'$ that is incident to the~$b$-edge and to the $c$-edge, making~$a$ the leaf closest to~$v'$ on that edge}. Similarly, if one edge contains at least two leaves~$a,b$ and another edge at least one leaf~$c$, then~$N$ is again~$\{a,d\}$-reconstructible for any~$d\in X\setminus \{a,b,c\}$.

A similar argument can be used to handle the case that the three edges form a triangle.

Finally, suppose that at least two of the three edges are multi-edges. Then, by Lemma~\ref{lem:bubble}, exactly two of the three edges form multi-edges, one of them containing two leaves, the other one no leaves, and the third edge of the three pairwise incident edges contains \leo{at least} one leaf. Then again it can be seen that~$N$ has leaf-reconstruction number at most~2 by using a similar argument as above.
\end{proof}

\begin{theorem}\label{thm:lev4}
Any binary level-4 phylogenetic network with at least five leaves \leo{has leaf-reconstruction number at most~2}.
\end{theorem}
\begin{proof}
Let~$N$ be such a network. By Corollary~\ref{cor:cutedge}, we may assume that~$N$ has no nontrivial cut-edges, i.e.~$N$ is simple.

If~$N$ is a simple level-1 network, pick any two~$x,y$ that are distance at least~4 apart. The fact that~$N$ is simple is $\{x,y\}$-recognizable. Moreover, using the fact that~$N$ has at least five leaves, it can easily be shown that~$N$ can be uniquely reconstructed from~$N_x$ and~$N_y$.

Now suppose that~$N$ is a simple level-$k$ network, with~$k\geq 2$.

If~$N$ has a 3-chain~$(x,y,z)$ and~$a\in X\setminus\{x,y,z\}$, then any $\{y,a\}$-reconstruction of~$N$ is simple. Moreover, by the proof of Lemma~\ref{lem:3chain} it can be concluded that~$N$ is $\{y,a\}$-reconstructible. Hence, we may assume that~$N$ contains no 3-chains.

If~$k=2$, then, considering the unique level-2 generator in Figure~\ref{fig:generators}, we are done by Lemma~\ref{lem:bubble}.

If~$k=3$, then there are two possible underlying generators, see Figure~\ref{fig:generators}. First suppose the underlying generator~$G$ is not~$K_4$ and thus has two pairs of multi-edges. Then, by Lemma~\ref{lem:bubble}, \leo{we may assume that} each pair of multi-edges has one edge containing exactly two leaves. Hence, we are done by Lemma~\ref{lem:4cycle}~(i). Now suppose that~$G=K_4$. Since~$|X|\geq 5$, it is straightforward to check that at least one 3-cycle~$C$ of~$G$ contains at least three leaves in~$N$. By Lemma~\ref{lem:4cycle}, it contains exactly~3 leaves. There are two cases (by Lemma~\ref{lem:3chain}). Either each edge of~$C$ contains exactly one leaf, or one edge contains two leaves and one edge one leaf. In either case, it is easy to check that wherever the other two leaves are, we can apply Lemma~\ref{lem:4cycle} to see that~$N$ has reconstruction number at most~2.

Finally, suppose~$k=4$. Then there are \leo{five} possibilities for the underlying generator~$G$, see Figure~\ref{fig:generators}. If~\leo{$G\in\{G_1,G_2,G_3\}$} then, by Lemma~\ref{lem:bubble}, each pair of multi-edges has one edge containing exactly two leaves and one edge containing no leaves. If~$G=G_1$ or~$G_3$, then we are done by Lemma~\ref{lem:4cycle}~(i). If~\leo{$G=G_2$}, then it is straightforward to check that, since~$|X|\geq 5$, there must exist some cycle that satisfies the condition of Lemma~\ref{lem:4cycle}~(ii).

Now suppose that~\leo{$G=G_4$}. Observe that~$G_4$ consists of two disjoint 3-cycles and three other edges, which we will call the \emph{middle edges}. For every vertex of~$G_4$, at most two edges incident to this vertex contain leaves by Lemma~\ref{lem:3incident}. Since~$|X|\geq 5$, it is straightforward to check that there is at least one vertex~$v$ of~$G_4$ with exactly two leaves~$a,b$ on the edges incident to~$v$.

First assume that~$a$ is on a middle edge and~$b$ is on a triangle edge. Then there is a unique Hamiltonian cycle~$C$ of~$G$ containing the~$a$-edge and the~$b$-edge. First suppose that there is at least one leaf~$c\in X\setminus \{a,b\}$ on an edge of~$C$. Assume that~$c$ is the first such leaf on the path along~$C$ between the neighbour of~$b$ and the neighbour of~$a$ not containing~$v$. Let~$i$ be the distance from the neighbour of~$b$ to the neighbour of~$c$ on this path. Let~$d\in X\setminus \{a,b,c\}$. Then $N$ is $\{c,d\}$-reconstructible, since the unique way to insert~$c$ into~$N_c$ is by attaching it to the $i$-th edge of the path along~$C$ from the neighbour of~$b$ to the neighbour of~$a$ not containing~$v$. Now suppose that none of the leaves in~$X\setminus\{a,b\}$ are on edges of~$C$. By Lemma~\ref{lem:3incident} there are no leaves on the third edge incident to~$v$. Hence, since~$|X|\geq 5$, there at least three leaves on the two edges of~$G$ that are not on~$C$ and not incident to~$v$. It is now straightforward to check that~$N$ has reconstruction number~2 by Lemma~\ref{lem:4cycle}~(i).

Now assume that~$a$ and~$b$ are both on the same triangle-edge. Then, if the previous case is not applicable for any vertex~$v'$ of~$G_4$, the only remaining possibility is that the other triangle also has an edge containg two leaves and we can apply Lemma~\ref{lem:4cycle}.

Now assume that~$a$ and~$b$ are on different triangle edges (of the same triangle). Then, if the previous cases are not applicable, all other leaves must be on the other triangle and we can use Lemma~\ref{lem:3incident}.

Finally, assume that~$a$ and~$b$ are both on the same middle edge. Then, if the previous cases are not applicable, the only remaining possibility is that some other middle edge also contains two leaves and we can apply Lemma~\ref{lem:4cycle}.

Now consider the last level-4 generator~\leo{$G_5=K_{3,3}$}. As before, it is straightforward to check that there is at least one vertex~$v$ of~$G_5$ with exactly two leaves~$a,b$ on the edges incident to~$v$.

First suppose that~$a$ and~$b$ are on different edges incident to~$v$. Observe that there are precisely two Hamiltonian cycles~$C$ and~$D$ of~$G_5$ containing the $a$-edge and the $b$-edge. Since each leaf is on an edge of at least one of~$C$ and~$D$, at least one edge of~$C$ and~$D$ contains a third leaf~$c\in X\setminus\{a,b\}$. Suppose that~$c$ is on an edge of~$C$. First suppose that all leaves are on edges of~$C$. Then we can use a similar argument as for the Hamiltonian cycle in~$G_4$ to show that~$N$ is $\{c,d\}$-reconstructible, for some~$d\in X\setminus\{a,b,c\}$. If at least one leaf~$e\in X\setminus\{a,b,c\}$ is on an edge that is not also on~$D$, then we choose the Hamiltonian cycle containing the~$e$-edge, and choose~$d\neq e$. Otherwise, all leaves are also on edges of~$D$. Observet that there are precisely four edges that are on both~$C$ and~$D$, which are two pairs of incident edges. Since~$|X|\geq 5$, it then follows by Lemma~\ref{lem:3incident} that~$N$ has leaf-reconstruction number~2. Now suppose that at least one leaf~$e\in X\setminus\{a,b,c\}$ is not on an edge of~$C$. Then~$N$ is $\{c,d\}$-reconstructible, with~$d\in X\setminus\{a,b,c,e\}$, again using a similar argument as for the Hamiltonian cycle in~$G_4$, choosing the Hamiltonian cycle of~$G$ not containing the~$e$-edge.

Finally, suppose that~$a$ and~$b$ are on the same edge incident to~$v$. Then, if the previous case is not applicable for any vertex~$v'$ of~$G_5$, the only remaining possibility is that there is some other edge of~$G_5$ containing two leaves and we can apply Lemma~\ref{lem:4cycle}~(ii). 
\end{proof}

\section{Reconstructing networks from quarnets}\label{sec:quarnets}

We have focussed so far on reconstructing networks from their $X$-deck. We could try to use a recursive argument in order to reconstruct networks from smaller subnetworks, with less than $|X|-1$ leaves. However, this approach does not work in general 
since there are networks for which no elements of its $X$-deck 
are phylogenetic networks, see Figure~\ref{fig:phylodeck}. Nevertheless, it is possible to apply a recursive approach if we use the following variant of the $X$-deck of a network.

\begin{definition}
Given a phylogenetic network~$N$ on~$X$ and a leaf~$x\in X$, the phylogenetic 
network~$N^\cP_x$ is the result of deleting leaf~$x$ from~$N$, together with its incident \leo{edge}, and applying the following \leo{three} operations until none is applicable:
\begin{itemize}
\item[(i)] suppress a degree-2 vertex;
\item[(ii)] replace a pair of multi-edges by a single edge;
\item[(iii)] collapse a \leo{blob} with precisely two incident cut-edges into a single vertex.
\end{itemize}
Given a phylogenetic network~$N$ on~$X$ and~$X'\subseteq X$, the \emph{phylogenetic} $X'$-\emph{deck} of~$N$ is the set $\{{N}^\cP_x \mid x\in X'\}$.
\end{definition}

See again Figure~\ref{fig:phylodeck} for an example.
Note that this form of leaf-deletion was introduced for directed level-1 phylogenetic networks in \cite{huber2011encoding} --
see also \cite{information} for more details for general phylogenetic networks. 

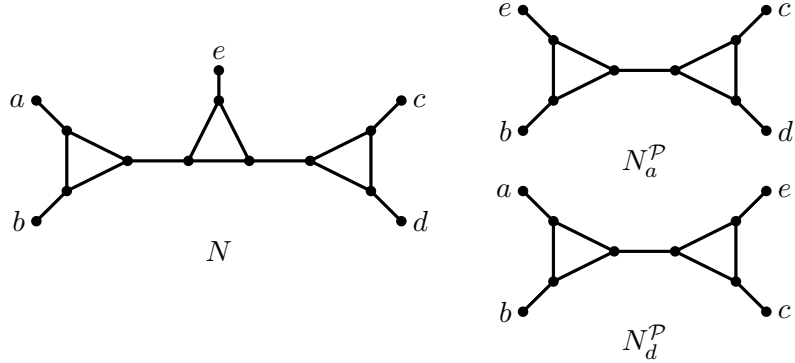
\begin{figure}
\centering
%\centerline{\includegraphics{fig/fig5}}
 \begin{tikzpicture}[scale=.8]
	 \tikzset{lijn/.style={very thick}}
	 \begin{scope}[xshift=0cm,yshift=0cm]
	\draw[very thick, fill, radius=0.06] (-.5,-.5) circle node[left] {$b$};
	\draw[very thick, fill, radius=0.06] (-.5,1.5) circle node[left] {$a$};
	\draw[very thick, fill, radius=0.06] (2.5,2) circle node[above] {$e$};
	\draw[very thick, fill, radius=0.06] (5.5,1.5) circle node[right] {$c$};
	\draw[very thick, fill, radius=0.06] (5.5,-.5) circle node[right] {$d$};
	\draw[very thick, fill, radius=0.06] (0,0) circle;
	\draw[very thick, fill, radius=0.06] (0,1) circle;
	\draw[very thick, fill, radius=0.06] (1,0.5) circle;
	\draw[very thick, fill, radius=0.06] (2,0.5) circle;
	\draw[very thick, fill, radius=0.06] (3,0.5) circle;
	\draw[very thick, fill, radius=0.06] (2.5,1.5) circle;
	\draw[very thick, fill, radius=0.06] (4,0.5) circle;
	\draw[very thick, fill, radius=0.06] (5,0) circle;
	\draw[very thick, fill, radius=0.06] (5,1) circle;
	\draw[lijn] (-0.5,-0.5) -- (0,0);
	\draw[lijn] (0,1) -- (0,0);
	\draw[lijn] (1,0.5) -- (0,0);
	\draw[lijn] (-0.5,1.5) -- (0,1);
	\draw[lijn] (1,0.5) -- (2,0.5);
	\draw[lijn] (2,0.5) -- (3,0.5);
	\draw[lijn] (2,0.5) -- (2.5,1.5);
	\draw[lijn] (2.5,1.5) -- (2.5,2);
	\draw[lijn] (3,0.5) -- (4,0.5);
	\draw[lijn] (4,0.5) -- (5,0);
	\draw[lijn] (4,0.5) -- (5,1);
	\draw[lijn] (5,0) -- (5,1);
	\draw[lijn] (5,0) -- (5.5,-0.5);
	\draw[lijn] (5,1) -- (5.5,1.5);
	\draw[lijn] (2.5,1.5) -- (3,0.5);
	\draw[lijn] (0,1) -- (1,0.5);
	\draw (2.5,-1) node {$N$};
	\end{scope}

\begin{scope}[xshift=8cm,yshift=1.5cm]
	\draw[very thick, fill, radius=0.06] (-.5,-.5) circle node[left] {$b$};
	\draw[very thick, fill, radius=0.06] (-.5,1.5) circle node[left] {$e$};
	\draw[very thick, fill, radius=0.06] (3.5,1.5) circle node[right] {$c$};
	\draw[very thick, fill, radius=0.06] (3.5,-.5) circle node[right] {$d$};
	\draw[very thick, fill, radius=0.06] (0,0) circle;
	\draw[very thick, fill, radius=0.06] (0,1) circle;
	\draw[very thick, fill, radius=0.06] (1,0.5) circle;
	\draw[very thick, fill, radius=0.06] (2,0.5) circle;
	\draw[very thick, fill, radius=0.06] (3,0) circle;
	\draw[very thick, fill, radius=0.06] (3,1) circle;
	\draw[lijn] (-0.5,-0.5) -- (0,0);
	\draw[lijn] (0,1) -- (0,0);
	\draw[lijn] (1,0.5) -- (0,0);
	\draw[lijn] (-0.5,1.5) -- (0,1);
	\draw[lijn] (1,0.5) -- (2,0.5);
	\draw[lijn] (2,0.5) -- (3,0);
	\draw[lijn] (2,0.5) -- (3,1);
	\draw[lijn] (3,0) -- (3,1);
	\draw[lijn] (3,0) -- (3.5,-0.5);
	\draw[lijn] (3,1) -- (3.5,1.5);
	\draw[lijn] (0,1) -- (1,0.5);
	\draw (1.5,-1) node {$N_a^\cP$};
	\end{scope}
	
	\begin{scope}[xshift=8cm,yshift=-1.5cm]
	\draw[very thick, fill, radius=0.06] (-.5,-.5) circle node[left] {$b$};
	\draw[very thick, fill, radius=0.06] (-.5,1.5) circle node[left] {$a$};
	\draw[very thick, fill, radius=0.06] (3.5,1.5) circle node[right] {$e$};
	\draw[very thick, fill, radius=0.06] (3.5,-.5) circle node[right] {$c$};
	\draw[very thick, fill, radius=0.06] (0,0) circle;
	\draw[very thick, fill, radius=0.06] (0,1) circle;
	\draw[very thick, fill, radius=0.06] (1,0.5) circle;
	\draw[very thick, fill, radius=0.06] (2,0.5) circle;
	\draw[very thick, fill, radius=0.06] (3,0) circle;
	\draw[very thick, fill, radius=0.06] (3,1) circle;
	\draw[lijn] (-0.5,-0.5) -- (0,0);
	\draw[lijn] (0,1) -- (0,0);
	\draw[lijn] (1,0.5) -- (0,0);
	\draw[lijn] (-0.5,1.5) -- (0,1);
	\draw[lijn] (1,0.5) -- (2,0.5);
	\draw[lijn] (2,0.5) -- (3,0);
	\draw[lijn] (2,0.5) -- (3,1);
	\draw[lijn] (3,0) -- (3,1);
	\draw[lijn] (3,0) -- (3.5,-0.5);
	\draw[lijn] (3,1) -- (3.5,1.5);
	\draw[lijn] (0,1) -- (1,0.5);
	\draw (1.5,-1) node {$N_d^\cP$};
	\end{scope}
	\end{tikzpicture}
%\centerline{\includegraphics[width=0.6\textwidth]{phylodeck.jpg}}
\caption{\label{fig:phylodeck} An example of a level-1 phylogenetic network~$N$ \leo{on~$X$} such that no elements of its \leo{$X$-}deck are phylogenetic networks. Nevertheless, it is possible to reconstruct~$N$ from the 
quarnets~$N_a^\cP$ and~$N_d^\cP$.}
\end{figure}

All elements of a phylogenetic $X$-deck are phylogenetic networks by the following observation, which is easily verified.

\begin{observation}
Let~$N$ be a phylogenetic network~$N$ on~$X$, \leo{with~$|X|\geq 3$}, and~$x\in X$. Then~${N}^\cP_x$ is a phylogenetic network on~$X\setminus\{x\}$.
\end{observation}

This opens the door to reconstructing networks from smaller subnetworks. A \emph{quarnet} is a phylogenetic network with precisely four leaves. The set of quarnets~$Q(N)$ of a phylogenetic network~$N$ on~$X$ is defined recursively by $Q(N)=\{N\}$ if~$|X|=4$ and 
\[Q(N)=\bigcup_{x\in X} Q({N}^\cP_x) \quad \text{if}~|X|\geq 5.
\]

\leo{Here, the union operation keeps one phylogenetic network from each group of equivalent phylogenetic networks. We say that two sets~$\cN,\cN'$ of phylogenetic networks are \emph{equivalent}, denoted $\cN\sim\cN'$, if there exists a bijection~$f:\cN\rightarrow \cN'$ with~$N\sim f(N)$ for all~$N\in\cN$.}

We say that a network~$N$ is \emph{reconstructible from its quarnets} if every phylogenetic network~$N'$ with $Q(N)\leo{\sim}Q(N')$ is equivalent to~$N$. Moreover, a class~$\cN$ of phylogenetic networks is \emph{quarnet\leo{-}reconstructible} if each~$N\in\cN$ is reconstructible from its quarnets.

Similarly,~$N$ is \emph{reconstructible from its phylogenetic~$X$-deck} if every phylogenetic network~$N'$, \leo{whose phylogenetic~$X$-deck is equivalent to the} phylogenetic~$X$-deck \leo{of}~$N$, is equivalent to~$N$. Moreover, a class~$\cN$ of phylogenetic networks is \emph{phylogenetically reconstructible} if each~$N\in\cN$ is reconstructible from its phylogenetic~$X$-deck.

\leo{If two phylogenetic networks on~$X$ have equivalent $X$-decks, then they have equivalent phylogenetic $X$-decks (but not conversely, see Figure~\ref{fig:phyloXdeck}). Consequently, if a phylogenetic network on~$X$ is reconstructible from its phylogenetic $X$-deck, then it is $X$-reconstructible. The following proposition, which shows that the converse is also true in some cases, will permit us to apply results from previous sections.}

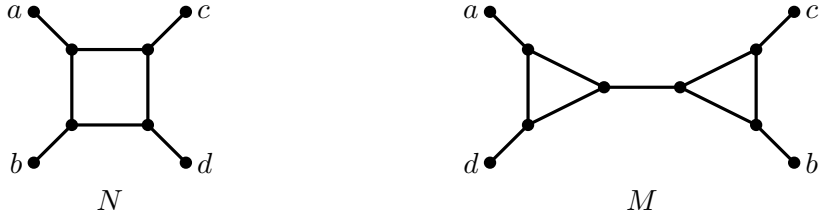
\begin{figure}
\centering
 \begin{tikzpicture}
	 \tikzset{lijn/.style={very thick}}
	% level-1 network 1
	\begin{scope}[xshift=0cm,yshift=0cm]
	\draw[very thick, fill, radius=0.06] (0.5,0.5) circle node[left] {$b$};
	\draw[very thick, fill, radius=0.06] (2.5,0.5) circle node[right] {$d$};
	\draw[very thick, fill, radius=0.06] (1,1) circle;
	\draw[very thick, fill, radius=0.06] (2,1) circle;
	\draw[very thick, fill, radius=0.06] (1,2) circle;
	\draw[very thick, fill, radius=0.06] (2,2) circle;
	\draw[very thick, fill, radius=0.06] (0.5,2.5) circle node[left] {$a$};
	\draw[very thick, fill, radius=0.06] (2.5,2.5) circle node[right] {$c$};
	\draw[lijn] (0.5,0.5) -- (1,1);
	\draw[lijn] (0.5,2.5) -- (1,2);
	\draw[lijn] (1,1) -- (2,1);
	\draw[lijn] (1,2) -- (2,2);
	\draw[lijn] (1,1) -- (1,2);
	\draw[lijn] (2,1) -- (2,2);
	\draw[lijn] (2,1) -- (2.5,0.5);
	\draw[lijn] (2,2) -- (2.5,2.5);
	\draw (1.5,0) node {$N$};
	\end{scope}
	
	% level-1 network 2
	\begin{scope}[xshift=6cm,yshift=0cm]
	\draw[very thick, fill, radius=0.06] (0.5,0.5) circle node[left] {$d$};
	\draw[very thick, fill, radius=0.06] (4.5,0.5) circle node[right] {$b$};
	\draw[very thick, fill, radius=0.06] (1,1) circle;
	\draw[very thick, fill, radius=0.06] (2,1.5) circle;
	\draw[very thick, fill, radius=0.06] (1,2) circle;
	\draw[very thick, fill, radius=0.06] (3,1.5) circle;
	\draw[very thick, fill, radius=0.06] (4,1) circle;
	\draw[very thick, fill, radius=0.06] (4,2) circle;
	\draw[very thick, fill, radius=0.06] (0.5,2.5) circle node[left] {$a$};
	\draw[very thick, fill, radius=0.06] (4.5,2.5) circle node[right] {$c$};
	\draw[lijn] (0.5,0.5) -- (1,1);
	\draw[lijn] (0.5,2.5) -- (1,2);
	\draw[lijn] (1,1) -- (1,2);
	\draw[lijn] (1,1) -- (2,1.5);
	\draw[lijn] (1,2) -- (2,1.5);
	\draw[lijn] (2,1.5) -- (3,1.5);
	\draw[lijn] (3,1.5) -- (4,1);
	\draw[lijn] (3,1.5) -- (4,2);
	\draw[lijn] (4,1) -- (4,2);
	\draw[lijn] (4,1) -- (4.5,0.5);
	\draw[lijn] (4,2) -- (4.5,2.5);
	\draw (2.5,0) node {$M$};
	\end{scope}

		\end{tikzpicture}
\caption{\label{fig:phyloXdeck} \leo{Two phylogenetic networks that have the same phylogenetic $X$-deck but not the same $X$-deck (even though the $X$-deck and phylogenetic $X$-deck of~$N$ are equivalent). Network~$N$ is neither $X$-reconstructible nor reconstructible from its phylogenetic $X$-deck, while~$M$ is $X$-reconstructible but not reconstructible from its phylogenetic $X$-deck.}}
\end{figure}

\begin{proposition}\label{prop:phylodeck}
\leo{Let~$N$ be a phylogenetic network on~$X$ with~$|X|\geq 4$.  If~$N$ is $Y$-reconstructible for some~$Y\subseteq X$ with~$|Y|\geq 2$ and~$N^\cP_y\sim N_y$ for all~$y\in Y$, then~$N$ is reconstructible from its phylogenetic~$X$-deck.}
\end{proposition}
\begin{proof}
\leo{Suppose that there exists a network~$M$ that is not equivalent to~$N$ but has an equivalent phylogenetic~$X$-deck. Since~$N$ is $Y$-reconstructible, there exists a~$y\in Y$ such that $N_y\not\sim M_y$. Since $M^\cP_y\sim N^\cP_y \sim N_y$, it follows that $M^\cP_y\not\sim M_y$ and hence that the neighbour of~$y$ in~$M$ is in a triangle. Moreover, since~$N_y$ has the same reticulation number as~$N$,~$M^\cP_y$ also has the same reticulation number as~$N$. Since, in~$M$, the neighbour of~$y$ is in a triangle,~$M$ has a higher reticulation number than~$M^\cP_y$ and~$N$. Take any~$z\in Y\setminus\{y\}$. Then, since $M^\cP_z\sim N^\cP_z \sim N_z$, $M^\cP_z$ has the same reticulation number as~$N$ and $M^\cP_y$ and hence a lower reticulation number than~$M$. It follows that the neighbour of~$z$ in~$M$ is also in a triangle. We distingish two cases.}

\leo{First assume that the neighbours of~$y$ and~$z$ are both in the same triangle in~$M$. Consider any two leaves~$x,p\in X\setminus\{y,z\}$. Then, the neighbours of~$y$ and~$z$ are together in the same triangle in~$M^\cP_x\sim N^\cP_x$ and in~$M^\cP_p\sim N^\cP_p$. On the other hand, neither of the neighbours of~$y$ and~$z$ is in a triangle in~$N$, since $N^\cP_z\sim N_z$ and $N^\cP_y\sim N_y$. This is only possible when~$N$ is a simple level-1 network on $X=\{x,y,z,p\}$. This contradicts the assumption that~$N$ is $Y$-reconstructible, with~$Y\subseteq X$, and hence~$X$-reconstructible.}

\leo{Now assume that the neighbours of~$y$ and~$z$ are in different triangles in~$M$. Then, the neighbour of~$z$ is also in a triangle in~$M^\cP_y \sim N_y$. On the other hand, the neighbour of~$z$ is not in a triangle in~$N$, since $N^\cP_z\sim N_z$. Hence, in~$N$, the neighbours of~$y$ and~$z$ are part of a 4-cycle. Consider again two leaves~$x,p\in X\setminus \{y,z\}$. In $N^\cP_x\sim M^\cP_x$ and in~$N^\cP_p\sim M^\cP_p$, the neighbours of~$y$ and~$z$ are in a triangle or 4-cycle. This is only possible when, in~$M$, the neighbours of (without loss of generality)~$x$ and~$y$ are in one triangle while the neighbours of~$p$ and~$z$ are in a different triangle, and the two triangles are adjacent. This implies that there are no other leaves, i.e. $X=\{x,y,z,p\}$, and again~$N$ is a simple level-1 network on~$X$. This again leads to a contradiction since~$N$ is $X$-reconstructible.}
\end{proof}

\leo{In particular, we have the following.}

\begin{corollary}\label{cor:xdecknetworks} \leo{Let~$N$ be a phylogenetic network on~$X$ with~$|X|\geq 4$. If the $X$-deck of~$N$ consists of only phylogenetic networks, then~$N$ is reconstructible from its phylogenetic $X$-deck if and only if~$N$ is $X$-reconstructible.}
\end{corollary}

\leo{Note that Corollary~\ref{cor:xdecknetworks} does not hold when~$|X|=3$, see Figure~\ref{fig:phyloXdeck2}.}

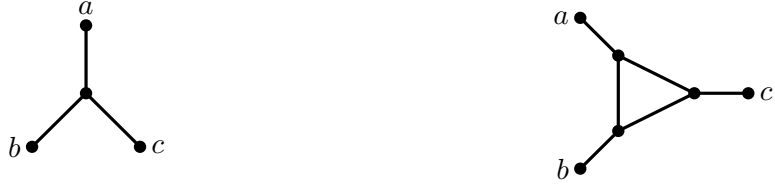
\begin{figure}
\centering
 \begin{tikzpicture}
	 \tikzset{lijn/.style={very thick}}
	% tree
	\begin{scope}[xshift=0cm,yshift=0cm]
	\draw[very thick, fill, radius=0.06] (0,0.9) circle node[above] {$a$};
	\draw[very thick, fill, radius=0.06] (-0.71,-0.71) circle node[left] {$b$};
	\draw[very thick, fill, radius=0.06] (0,0) circle;
	\draw[very thick, fill, radius=0.06] (0.71,-0.71) circle node[right] {$c$};
	\draw[lijn] (0,0) -- (0,0.9);
	\draw[lijn] (0,0) -- (0.71,-0.71);
	\draw[lijn] (0,0) -- (-0.71,-0.71);
	%\draw (0,-1.5) node {$N$};
	\end{scope}
	
	% level-1 network 
	\begin{scope}[xshift=6cm,yshift=-1.5cm]
	\draw[very thick, fill, radius=0.06] (0.5,0.5) circle node[left] {$b$};
	\draw[very thick, fill, radius=0.06] (1,1) circle;
	\draw[very thick, fill, radius=0.06] (2,1.5) circle;
	\draw[very thick, fill, radius=0.06] (1,2) circle;
	\draw[very thick, fill, radius=0.06] (2.71,1.5) circle node[right] {$c$};
	\draw[very thick, fill, radius=0.06] (0.5,2.5) circle node[left] {$a$};
	\draw[lijn] (0.5,0.5) -- (1,1);
	\draw[lijn] (0.5,2.5) -- (1,2);
	\draw[lijn] (1,1) -- (1,2);
	\draw[lijn] (1,1) -- (2,1.5);
	\draw[lijn] (1,2) -- (2,1.5);
	\draw[lijn] (2,1.5) -- (2.71,1.5);
	%\draw (1.5,0) node {$M$};
	\end{scope}

		\end{tikzpicture}
\caption{\label{fig:phyloXdeck2} \leo{Phylogenetic networks on~$X=\{a,b,c\}$ that are $X$-reconstructible but not reconstructible from their phylogenetic $X$-deck.}}
\end{figure}

\begin{theorem}\label{thm:recursion}
Let~$\cN$ be a class of phylogenetic networks such \leo{that} each element of~$\cN$ has at least five leaves and, for each element~$N$ of~$\cN$ with at least six leaves, the phylogenetic~$X$-deck of~$N$ is \leo{equivalent to} a subset of~$\cN$. Then~$\cN$ is phylogenetically-reconstructible if and only if it is quarnet-reconstructible.
\end{theorem}
\begin{proof}
If~$\cN$ is quarnet-reconstructible then it is phylogenetically-reconstructible since if two phylogenetic networks~$N,N'\in \cN$ have \leo{equivalent} phylogenetic $X$-deck\leo{s} then \leo{it follows directly that}~$Q(N)\leo{\sim}Q(N')$.

Now suppose that~$\cN$ is phylogenetically-reconstructible. We prove by induction on~$i$ that each~$N\in\cN$ with at most~$i$ leaves is \leo{quarnet-reconstructible}. If~$i=5$ then the phylogenetic $X$-deck of~$N$ is equal to~$Q(N)$ and therefore~$N$ is \leo{quarnet-reconstructible}. Now suppose~$i\geq 6$. Since~$N$ is reconstructible from its~$X$-deck and each element of its $X$-deck is, by induction,  \leo{quarnet-reconstructible}, $N$ is  \leo{quarnet-reconstructible}.
\end{proof}

First observe that each phylogenetic tree on~$X$ with~$|X|\geq 5$ is reconstructible from its phylogenetic $X$-deck by Theorem~\ref{thm:trees} \leo{and Proposition~\ref{prop:phylodeck}}. Hence, the class of phylogenetic trees with at least five leaves is phylogenetically reconstructible.

However, a similar argument cannot be used to show that even the class of level-1 networks is phylogenetically reconstructible. Therefore, it is interesting to study which classes of networks are phylogenetically reconstructible.

%We are now interested to understand which classes of networks are phylogenetically reconstructible. Note that we cannot immediately apply the previous results since we have to avoid the situation that the removal of a leaf makes a large part of the network to collapse.

\begin{theorem}\label{thm:lev3phylo}
The class of level-3 phylogenetic networks with at least five leaves is phylogenetically reconstructible.
\end{theorem}

To prove this theorem, we will first show that an analogue of Theorem~\ref{thm:cutedge} holds.

\begin{theorem}\label{thm:cutedgephylo}
The class of decomposable phylogenetic networks with at least five leaves is phylogenetically reconstructible.
\end{theorem}
\begin{proof}
The proof is very similar to that of Theorem~\ref{thm:cutedge}. As in that proof, first note that a phylogenetic network has at least one nontrivial cut-edge if and only if at most two elements of its phylogenetic $X$-deck do not. Let~$N$ be some phylogenetic network on~$X$ with at least one nontrivial cut-edge and~$|X|\geq 5$. Since $({T(N)})^\cP_x = T({N}^\cP_x)$, \leo{for all~$x\in X$}, we can reconstruct~$T(N)$ from the phylogenetic $X$-deck of~$N$. We can then use exactly the same argument as in the last part of the proof \leo{of} Theorem~\ref{thm:cutedge} to show that~$N$ is reconstructible from its phylogenetic $X$-deck (see Figure~\ref{fig:phylodeck} for an illustration).
\end{proof}

We now prove Theorem~\ref{thm:lev3phylo}.
\begin{proof}
By Theorem~\ref{thm:cutedgephylo}, it suffices to consider simple level-$k$ networks with~$1\leq k\leq 3$. For simple level-1 networks, the phylogenetic $X$-deck is precisely equal to the $X$-deck and we are done \leo{by Proposition~\ref{prop:phylodeck}.}

Now consider a simple level-2 network~$N$ and its underlying generator~$G$. If the phylogenetic $X$-deck of~$N$ is not equal to its $X$-deck then one of the three edges of~$G$ contains exactly one leaf~$x$, another edge of~$G$ contains no leaves, and the third edge of~$G$ contains all other leaves~$X\setminus\{x\}$. Then~$N$ is $\{y,z\}$-reconstructible for any $y,z\in X\setminus\{x\}$ with distance between them at least~4. Since $N_y^\cP=N_y$ and $N_z^\cP=N_z$ we are done \leo{by Proposition~\ref{prop:phylodeck}.}

Therefore, we may assume that~$N$ is a simple level-3 network. Suppose the phylogenetic $X$-deck of~$N$ is not equal to its $X$-deck. Then the underlying generator~$G$ of~$N$ is not equal to~$K_4$ (since~$K_4$ does not have any multi-edges). Hence,~$G$ is the other level-3 generator, see Figure~\ref{fig:generators}. Moreover, at least one pair of multi-edges contains precisely one leaf, say leaf~$x$. The other pair of multi-edges contains at least one leaf~$y$.

If there is at least one leaf~$z$ on an edge that is not in a pair of multi-edges, then it is straightforward to check that, wherever you put leaves~$p,q\in X\setminus\{x,y,z\}$, there is a cycle containing the neighbours of leaves~$a,b,c,d$ satisfying the conditions of Lemma~\ref{lem:4cycle}(i) and a fifth leaf~$e$ such that $N_d^\cP=N_d$ and $N_e^\cP=N_e$, \leo{and we are done by Proposition~\ref{prop:phylodeck}.}

The only remaining case is that all leaves in~$X\setminus\{x\}$ are on the pair of multi-edges not containing~$x$. Then there is again a cycle containing the neighbours of leaves~$a,b,c,d$ satisfying the conditions of Lemma~\ref{lem:4cycle}(i) and a fifth leaf~$e$ such that $N_d^\cP=N_d$. However, if~$|X|=5$ then the only choice for~$e$ is~$e=x$ and hence $N_e^\cP \leo{\not\sim} N_e$. Nevertheless, we can use a similar argument as in the proof of Lemma~\ref{lem:4cycle}(i) since $N_e^\cP$ does contain a unique cycle containing the neighbours of~$a,b,c$ and~$d$.
\end{proof}

\begin{corollary}\label{cor:lev3quarnets}
Any level-3 phylogenetic network is reconstructible from its quarnets.
\end{corollary}

\section{Edge-reconstructibility}\label{sec:idge}

In this section we shall consider the problem of reconstructing a phylogenetic network from
its edge-deleted networks. We first formalize this concept (cf. \cite[Section 2]{bh77} for a review of edge-reconstruction in graphs).

Given a phylogenetic network~$N$ and an edge~$e\in E(N)$, the pseudo-network~$N_e$ is the result of deleting edge~$e$ from~$N$ and suppressing resulting degree-2 vertices. The \emph{edge-deck} of~$N$ is the multiset $\{N_e \mid e\in E(N)\}$. An \emph{edge-reconstruction} of a network~$N$ on~$X$ is a network~$N'$ on~$X$ with~$E(N')=E(N)$ and~$N'_e\sim N_e$ for all~$e\in E(N)$. \leo{Note that by $E(N')=E(N)$ we do not mean that the edges of~$N$ are the same pairs of vertices as the edges of~$N'$, but that there exists a bijection $f:E(N)\rightarrow E(N')$ which we assume to be the identity.} We call a phylogenetic network~$N$ \emph{edge-reconstructible} if every edge-reconstruction of~$N$ is equivalent to~$N$.

\begin{lemma}\label{lem:leafedge}
Let~$N$ be a phylogenetic network on~$X$. If~$N$ is leaf-reconstructible then it is edge-reconstructible.
\end{lemma}
\begin{proof}
This follows directly from the observation that $N_e \sim N'_e$ if and only if $N_x \sim N'_x$ for each edge~$e$ that has \leo{an} endpoint~$x\in X$ \leo{in both~$N$ and~$N'$}.
\end{proof}

However, there exist edge-reconstructible networks that are not leaf-reconstructible, see the examples in Figure~\ref{fig:notreconstructible}.

\begin{figure}
%\centerline{\includegraphics{fig/fig6}}
\centering
 \begin{tikzpicture}
	 \tikzset{lijn/.style={very thick}}
	 % tree 1
	\begin{scope}[xshift=0cm,yshift=-0.5cm]
	\draw[very thick, fill, radius=0.06] (0.5,0.5) circle node[left] {$b$};
	\draw[very thick, fill, radius=0.06] (2.5,0.5) circle node[right] {$d$};
	\draw[very thick, fill, radius=0.06] (1,1) circle;
	\draw[very thick, fill, radius=0.06] (2,1) circle;
	\draw[very thick, fill, radius=0.06] (0.5,1.5) circle node[left] {$a$};
	\draw[very thick, fill, radius=0.06] (2.5,1.5) circle node[right] {$c$};
	\draw[lijn] (0.5,0.5) -- (1,1);
	\draw[lijn] (0.5,1.5) -- (1,1);
	\draw[lijn,dashed] (1,1) -- (2,1);
	\draw[lijn] (2,1) -- (2.5,0.5);
	\draw[lijn] (2,1) -- (2.5,1.5);
	\end{scope}
	
	% tree 2
	\begin{scope}[xshift=6cm,yshift=-0.5cm]
	\draw[very thick, fill, radius=0.06] (0.5,0.5) circle node[left] {$d$};
	\draw[very thick, fill, radius=0.06] (2.5,0.5) circle node[right] {$b$};
	\draw[very thick, fill, radius=0.06] (1,1) circle;
	\draw[very thick, fill, radius=0.06] (2,1) circle;
	\draw[very thick, fill, radius=0.06] (0.5,1.5) circle node[left] {$a$};
	\draw[very thick, fill, radius=0.06] (2.5,1.5) circle node[right] {$c$};
	\draw[lijn] (0.5,0.5) -- (1,1);
	\draw[lijn] (0.5,1.5) -- (1,1);
	\draw[lijn] (1,1) -- (2,1);
	\draw[lijn] (2,1) -- (2.5,0.5);
	\draw[lijn] (2,1) -- (2.5,1.5);
	\end{scope}

	% level-1 network 1
	\begin{scope}[xshift=0cm,yshift=-3.5cm]
	\draw[very thick, fill, radius=0.06] (0.5,0.5) circle node[left] {$b$};
	\draw[very thick, fill, radius=0.06] (2.5,0.5) circle node[right] {$d$};
	\draw[very thick, fill, radius=0.06] (1,1) circle;
	\draw[very thick, fill, radius=0.06] (2,1) circle;
	\draw[very thick, fill, radius=0.06] (1,2) circle;
	\draw[very thick, fill, radius=0.06] (2,2) circle;
	\draw[very thick, fill, radius=0.06] (0.5,2.5) circle node[left] {$a$};
	\draw[very thick, fill, radius=0.06] (2.5,2.5) circle node[right] {$c$};
	\draw[lijn] (0.5,0.5) -- (1,1);
	\draw[lijn] (0.5,2.5) -- (1,2);
	\draw[lijn,dashed] (1,1) -- (2,1);
	\draw[lijn] (1,2) -- (2,2);
	\draw[lijn] (1,1) -- (1,2);
	\draw[lijn] (2,1) -- (2,2);
	\draw[lijn] (2,1) -- (2.5,0.5);
	\draw[lijn] (2,2) -- (2.5,2.5);
	\end{scope}
	
	% level-1 network 2
	\begin{scope}[xshift=6cm,yshift=-3.5cm]
	\draw[very thick, fill, radius=0.06] (0.5,0.5) circle node[left] {$d$};
	\draw[very thick, fill, radius=0.06] (2.5,0.5) circle node[right] {$b$};
	\draw[very thick, fill, radius=0.06] (1,1) circle;
	\draw[very thick, fill, radius=0.06] (2,1) circle;
	\draw[very thick, fill, radius=0.06] (1,2) circle;
	\draw[very thick, fill, radius=0.06] (2,2) circle;
	\draw[very thick, fill, radius=0.06] (0.5,2.5) circle node[left] {$a$};
	\draw[very thick, fill, radius=0.06] (2.5,2.5) circle node[right] {$c$};
	\draw[lijn] (0.5,0.5) -- (1,1);
	\draw[lijn] (0.5,2.5) -- (1,2);
	\draw[lijn] (1,1) -- (2,1);
	\draw[lijn] (1,2) -- (2,2);
	\draw[lijn] (1,1) -- (1,2);
	\draw[lijn] (2,1) -- (2,2);
	\draw[lijn] (2,1) -- (2.5,0.5);
	\draw[lijn] (2,2) -- (2.5,2.5);
	\end{scope}
	
	% level-2 network 1
	\begin{scope}[xshift=0cm,yshift=-6cm]
	\draw[very thick, fill, radius=0.06] (0,0) circle node[below] {$b$};
	\draw[very thick, fill, radius=0.06] (0.5,0.5) circle;
	\draw[very thick, fill, radius=0.06] (1.5,0.5) circle;
	\draw[very thick, fill, radius=0.06] (2.5,0.5) circle;
	\draw[very thick, fill, radius=0.06] (1.5,0) circle node[below] {$c$};
	\draw[very thick, fill, radius=0.06] (3,0) circle node[below] {$d$};
	\draw[lijn] (0,0) -- (.5,.5);
	\draw[lijn] (1.5,0) -- (1.5,.5);
	\draw[lijn] (3,0) -- (2.5,.5);
	\draw[lijn] (1.5,0.5) -- (.5,.5);
	\draw[lijn,dashed] (1.5,0.5) -- (2.5,.5);
	\draw[very thick, fill, radius=0.06] (1,1) circle;
	\draw[very thick, fill, radius=0.06] (2,1) circle;
	\draw[lijn] (1,1) -- (2,1);
	\draw[lijn] (0.5,0.5) -- (1,1);
	\draw[lijn] (2.5,0.5) -- (2,1);
	\draw[very thick, fill, radius=0.06] (1.5,1.5) circle;
	\draw[very thick, fill, radius=0.06] (1.5,2) circle node[above] {$a$};
	\draw[lijn] (1,1) -- (1.5,1.5);
	\draw[lijn] (2,1) -- (1.5,1.5);
	\draw[lijn] (1.5,1.5) -- (1.5,2);
	\end{scope}

	% level-2 network 2
	\begin{scope}[xshift=6cm,yshift=-6cm]
	\draw[very thick, fill, radius=0.06] (0,0) circle node[below] {$b$};
	\draw[very thick, fill, radius=0.06] (0.5,0.5) circle;
	\draw[very thick, fill, radius=0.06] (1.5,0.5) circle;
	\draw[very thick, fill, radius=0.06] (2.5,0.5) circle;
	\draw[very thick, fill, radius=0.06] (1.5,0) circle node[below] {$d$};
	\draw[very thick, fill, radius=0.06] (3,0) circle node[below] {$c$};
	\draw[lijn] (0,0) -- (.5,.5);
	\draw[lijn] (1.5,0) -- (1.5,.5);
	\draw[lijn] (3,0) -- (2.5,.5);
	\draw[lijn] (1.5,0.5) -- (.5,.5);
	\draw[lijn] (1.5,0.5) -- (2.5,.5);
	\draw[very thick, fill, radius=0.06] (1,1) circle;
	\draw[very thick, fill, radius=0.06] (2,1) circle;
	\draw[lijn] (1,1) -- (2,1);
	\draw[lijn] (0.5,0.5) -- (1,1);
	\draw[lijn] (2.5,0.5) -- (2,1);
	\draw[very thick, fill, radius=0.06] (1.5,1.5) circle;
	\draw[very thick, fill, radius=0.06] (1.5,2) circle node[above] {$a$};
	\draw[lijn] (1,1) -- (1.5,1.5);
	\draw[lijn] (2,1) -- (1.5,1.5);
	\draw[lijn] (1.5,1.5) -- (1.5,2);
	\end{scope}
	
	% level-3 network 1
	\begin{scope}[xshift=0cm,yshift=-8.5cm]
	\draw[very thick, fill, radius=0.06] (0,0) circle node[below] {$a$};
	\draw[very thick, fill, radius=0.06] (0.5,0.5) circle;
	\draw[very thick, fill, radius=0.06] (1.5,0.5) circle;
	\draw[very thick, fill, radius=0.06] (2.5,0.5) circle;
	\draw[very thick, fill, radius=0.06] (1.5,0) circle node[below] {$b$};
	\draw[very thick, fill, radius=0.06] (3,0) circle node[below] {$c$};
	\draw[lijn] (0,0) -- (.5,.5);
	\draw[lijn] (1.5,0) -- (1.5,.5);
	\draw[lijn] (3,0) -- (2.5,.5);
	\draw[lijn] (1.5,0.5) -- (.5,.5);
	\draw[lijn, dashed] (1.5,0.5) -- (2.5,.5);
	\draw[very thick, fill, radius=0.06] (0.5,1) circle;
	\draw[very thick, fill, radius=0.06] (1.5,1) circle;
	\draw[very thick, fill, radius=0.06] (2.5,1) circle;
	\draw[lijn] (0.5,1) -- (2.5,1);
	\draw[lijn] (0.5,0.5) -- (0.5,1);
	\draw[lijn] (2.5,0.5) -- (2.5,1);
	\draw[very thick, fill, radius=0.06] (1.5,1.5) circle;
	\draw[lijn] (0.5,1) -- (1.5,1.5);
	\draw[lijn] (2.5,1) -- (1.5,1.5);
	\draw[lijn] (1.5,1.5) -- (1.5,1);
	\end{scope}
	
	% level-3 network 2
	\begin{scope}[xshift=6cm,yshift=-8.5cm]
	\draw[very thick, fill, radius=0.06] (0,0) circle node[below] {$a$};
	\draw[very thick, fill, radius=0.06] (0.5,0.5) circle;
	\draw[very thick, fill, radius=0.06] (1.5,0.5) circle;
	\draw[very thick, fill, radius=0.06] (2.5,0.5) circle;
	\draw[very thick, fill, radius=0.06] (1.5,0) circle node[below] {$c$};
	\draw[very thick, fill, radius=0.06] (3,0) circle node[below] {$b$};
	\draw[lijn] (0,0) -- (.5,.5);
	\draw[lijn] (1.5,0) -- (1.5,.5);
	\draw[lijn] (3,0) -- (2.5,.5);
	\draw[lijn] (1.5,0.5) -- (.5,.5);
	\draw[lijn] (1.5,0.5) -- (2.5,.5);
	\draw[very thick, fill, radius=0.06] (0.5,1) circle;
	\draw[very thick, fill, radius=0.06] (1.5,1) circle;
	\draw[very thick, fill, radius=0.06] (2.5,1) circle;
	\draw[lijn] (0.5,1) -- (2.5,1);
	\draw[lijn] (0.5,0.5) -- (0.5,1);
	\draw[lijn] (2.5,0.5) -- (2.5,1);
	\draw[very thick, fill, radius=0.06] (1.5,1.5) circle;
	\draw[lijn] (0.5,1) -- (1.5,1.5);
	\draw[lijn] (2.5,1) -- (1.5,1.5);
	\draw[lijn] (1.5,1.5) -- (1.5,1);
	\end{scope}

	\end{tikzpicture}
%\centerline{\includegraphics[width=.6\textwidth]{notreconstructible.jpg}}
\caption{\label{fig:notreconstructible} Pairs of phylogenetic networks that are not leaf-reconstructible but that are edge-reconstructible. The dashed edges indicate an edge~$e$ such that~$N_e$ is not contained in the edge-deck of the other network of the pair.}
\end{figure}
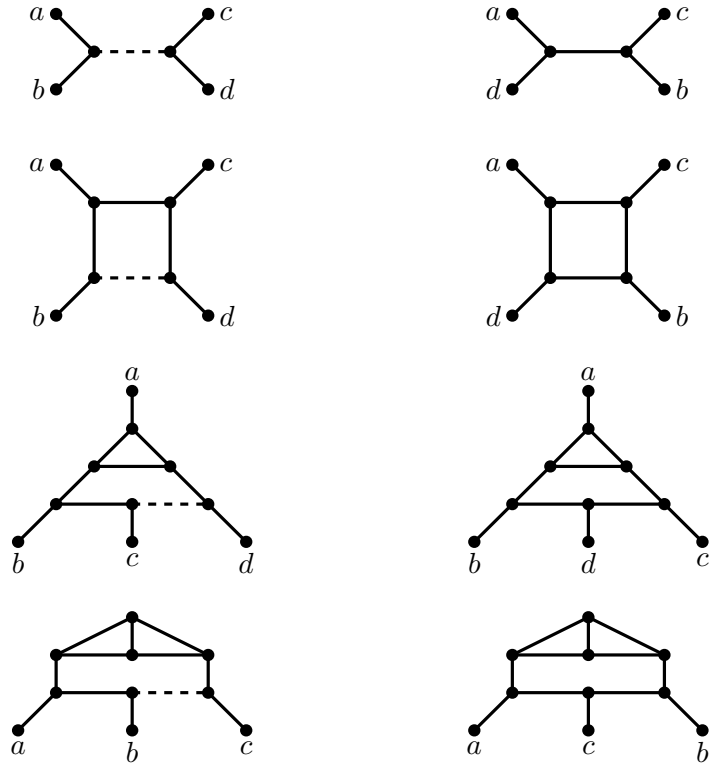

When considering edge-reconstructability of binary networks we can, by Theorem~\ref{thm:cutedge} and Lemma~\ref{lem:leafedge}, again restrict to simple networks.

We say that $(x,y)$ is a \emph{2-chain} of a phylogenetic network~$N$ on~$X$ if~$x,y\in X$ and \leo{the distance between~$x$ and~$y$ in~$N$ is~3}.

\begin{proposition}\label{prop:2chain}
Any simple binary phylogenetic network on~$X$ containing a 2-chain is edge-reconstructible.
\end{proposition}
\begin{proof}
\leo{The fact} that~$N$ \leo{is simple can} be recognized by considering three elements of its edge-deck~\leo{$N_{e_1},N_{e_2},N_{e_3}$} such that each of~$e_1,e_2,e_3$ is incident to a leaf. Since each of~$N_{e_1},N_{e_2},N_{e_3}$ \leo{consists of a simple network and an isolated vertex}, any edge-reconstruction of~$N$ is simple.

Suppose that~$N$ has a 2-chain $(x,y)$. Let~$u$ and~$v$ be the neighbours of~$x$ and~$y$ in~$N$ respectively and~\leo{$e=\{u,v\}$}. Let~$u'$ and~$v'$ be the neighbours of~$x$ and~$y$ in~$N_e$ respectively. 

First suppose that $(x,y)$ is not a 2-chain in $N_{e}$. There exists at least one \leo{edge}~$f$ that is not incident to~$u$ or~$v$. Since $(x,y)$ is a 2-chain in $N_{f}$, we can uniquely reconstruct~$N$ from~$N_{e}$ by subdividing the edges \leo{$\{u',x\}$ and $\{v',y\}$} and creating a new edge between the subdividing vertices.

Now suppose that $(x,y)$ is also a 2-chain in $N_{e}$. We say that a network has an $xy$-\emph{ladder} of \emph{length}~$k$ if there exist disjoint paths~$(x,u_1,\ldots ,u_k)$ and $(y,v_1,\ldots ,v_k)$ such that~$u_i$ and~$v_i$ are adjacent for~$1\leq i\leq k$. Let~$p\geq 1$ be the maximum length of an $xy$-ladder in~$N$. Take any such ladder and observe that there exists at least one \leo{edge}~$g$ that is not incident to any vertex of the ladder. Then the maximum length of an $xy$-ladder is~$p$ in~$N_g$ and is~$p-1$ in~$N_e$. Hence, we can again uniquely reconstruct~$N$ from~$N_{e}$ by subdividing the edges \leo{$\{u',x\}$ and $\{v',y\}$} and creating a new edge between the subdividing vertices.
\end{proof}

The following corollary can be proved in a similar way to Corollaries~\ref{cor:levk} and~\ref{cor:manyleaves}.
\clearpage

\begin{corollary}\label{cor:edgelevk}
\noindent\begin{itemize}
\item[(i)] Any simple binary level-$k$ phylogenetic network on~$X$ with~$k\geq 2 $ and~$|X|\geq 3k-2$ is edge-reconstructible.
\item[(ii)] Any binary phylogenetic network~$N=(V,E)$ on~$X$ with~$|X|\geq \min\{3(|E|-|V|)+1,5\}$ is edge-reconstructible.
\end{itemize}
\end{corollary}

\section{Discussion}\label{sec:discuss}

In this paper we have introduced the concept of leaf-recon\-structible phylogenetic networks.
We have shown that several large classes of phylogenetic networks are leaf-reconstructible, and 
used our results to show that level-3 networks are defined by their quarnets.
\leo{We conjecture that all unrooted phylogenetic networks with 5 or more leaves are leaf-reconstructible. We expect that this could be a difficult conjecture to settle, as with
other variants of the graph reconstruction conjecture.}

In another direction, it could be of interest to also consider leaf-reconstructibility of
nonbinary networks. In Theorem~\ref{thm:trees}, we showed that nonbinary phylogenetic trees are leaf-reconstructible, \leo{and in Theorem~\ref{thm:cutedge} that even all decomposable nonbinary phylogenetic networks are leaf-reconstructible, but what about non-decomposable nonbinary networks?}
%To address this question it would
%first be useful to know whether or not the number of vertices, edges, reticulation number and level of 
%a nonbinary network recognizable. More genearally, to shed light on 
%the difficulty of deciding leaf-reconstructibility for nonbinary networks
\leo{The following related question} could also be worth considering:
If every nonbinary phylogenetic network with at least five leaves is leaf-reconstructible, then is every graph reconstructible?  

\leo{In Section~\ref{sec:idge}}, we considered edge-reconstructibility, a variant of the leaf-recon\-struc\-tibility problem.
Another variant that should be considered is leaf-reconstructibility for
directed phylogenetic networks. This is an important class of networks, in which the 
networks are directed acyclic graphs, with a single root and leaves labeled by the set~\leo{$X$}. 
In \cite{information} certain examples of directed phylogenetic 
networks are presented which indicate that such networks may not be leaf-reconstructible\leo{, but it remains an open problem whether or not this is the case (note that not all digraphs are reconstructible~\cite{stock77}).}

In the longer term, it would be interesting to consider leaf-reconstructibility of
networks that arise in biological settings. Indeed, even if not every network is
leaf-reconstructible, it may be that counter-examples are somewhat unlikely to occur
as evolutionary histories (e.g. if they are highly symmetric). 

One way to approach this could be to consider random networks.
As we have seen in Corollary~\ref{cor:almostall}, for any 
fixed~$k$, almost all level-$k$ phylogenetic networks are leaf-reconstructible.
It would be interesting to know whether or not almost all phylogenetic networks on a fixed leaf-set 
are leaf-reconstructible. In this context, it is 
worth noting that almost every graph has reconstructing number three \cite{b90}.
We have shown that \leo{decomposable and \leo{binary} level-4 networks with at least five leaves have reconstruction number at most~2}. So, do almost all \leo{(binary)} phylogenetic networks have reconstruction number \leo{at most}~2?  

\leo{Finally}, it would be interesting 
to consider leaf-reconstructibilty of
networks that are generated accord\leo{ing} to some model of molecular evolution (see e.g.~\cite{felsenstein} for a review of such models). This would 
be somewhat analogous to recent ground-breaking 
work on reconstructibility of pedigrees in a 
stochastic setting \cite{thatte11,thatte07}, and could focus on models such 
as those presented in, for example, \cite{luay}.

\section*{Acknowledgements}
Part of this work was conducted while Vincent Moulton was visiting the TU Delft on visitors grant 040.11.529 funded by the Netherlands Organization for Scientific Research (NWO). Leo van Iersel was partially supported by NWO, including Vidi grant 639.072.602, and partially by the 4TU Applied Mathematics Institute.

\bibliographystyle{siamplain}
\providecommand{\bysame}{\leavevmode\hbox to3em{\hrulefill}\thinspace}
\providecommand{\MR}{\relax\ifhmode\unskip\space\fi MR }
\providecommand{\MRhref}[2]{%
  \href{http://www.ams.org/mathscinet-getitem?mr=#1}{#2}
}
\providecommand{\href}[2]{#2}

\end{document}